\documentclass[a4paper]{article}%
\usepackage{amsfonts}
\usepackage{amsmath}
\usepackage{amssymb}
\usepackage{graphicx}
\usepackage[a4paper]{geometry}
\usepackage{amstext}%
\providecommand{\U}[1]{\protect\rule{.1in}{.1in}}
\newtheorem{theorem}{Theorem}

\newtheorem{corollary}[theorem]{Corollary}

\newtheorem{lemma}[theorem]{Lemma}

\newtheorem{proposition}[theorem]{Proposition}
\newtheorem{remark}[theorem]{Remark}

\newenvironment{proof}[1][Proof]{\noindent\textbf{#1.} }{\ \rule{0.5em}{0.5em}}
\begin{document}

\title{Stability and control of a 1D quantum system with confining time dependent
delta potentials.}
\author{Andrea Mantile\thanks{IRMAR, UMR - CNRS 6625, Universit\'{e} Rennes 1, Campus
de Beaulieu, 35042 Rennes Cedex, France.}}
\date{}
\maketitle

\begin{abstract}
The evolution problem for a quantum particle confined in a 1D box and
interacting with one fixed point through a time dependent point interaction is
considered. Under suitable assumptions of regularity for the time profile of
the Hamiltonian, we prove the existence of strict solutions to the
corresponding Schr\"{o}dinger equation. The result is used to discuss the
stability and the steady-state local controllability of the wavefunction when
\ the strenght of the interaction is used as a control parameter.

\end{abstract}

\section{Introduction}

We consider the evolution problem for a quantum particle confined in a 1D
bounded domain and interacting with one fixed point through a delta shaped
potential whose strength varies in time. The Hamiltonian associated to this
class of potentials, denoted in the following with $H_{\alpha(t)}$, is defined
in terms of a time dependent parameter, $\alpha(t)$, which describes the time
profile of the interaction. Under suitable regularity assumptions, we study
the dynamical system defined by $H_{\alpha(t)}$, its stability properties and
the local controllability of the dynamics when $\alpha$ is used as a control function.

General conditions for the solution to the quantum evolution equation related
to non autonomous Hamiltonians, $H(t)$, have been long time investigated (e.g.
in \cite{Simon}, \cite{Reed} and \cite{Fattorini}). When the operator's domain
$D(H(t))$ depends on time (as in the case of time dependent point
interactions), the Cauchy problem%
\begin{equation}
\left\{
\begin{array}
[c]{l}%
\frac{d}{dt}\psi=-iH(t)\psi\\
\psi_{t=0}=\psi_{0}%
\end{array}
\right.  , \label{Cauchy}%
\end{equation}
was explicitely considered in \cite{Kys} by Kisy\'{n}ski using coercivity and
$C_{loc}^{2}$-regularity of $t\rightarrow H(t)$. A similar assumption,
$\alpha\in C_{loc}^{2}(t_{0},+\infty)$ in the above notation, is used in the
works of Yafaev \cite{Yaf1}, \cite{Yaf2} and \cite{Yaf3} (with M.~Sayapova) to
prove the existence of a strongly differentiable time propagator for the
scattering problem with time dependent delta interactions in $\mathbb{R}^{3}$.
Such a condition, however, can be relaxed by exploiting the explicit character
of point interactions. The dynamics associated to this class of operators,
indeed, is essentially described by the evolution of a finite dimensional
variable related to the values taken by the regular part of the state in the
interaction points (e.g. in \cite{Albeverio}, \cite{Dell'Anto}, \cite{AdTe}).
This leads to simplified evolution equations allowing rather explicit
estimates under Fourier transform. Using this approach, the quantum evolution
problem for a 1D time dependent delta interaction has been considered in
\cite{Taoufik}. In particular, it is shown that $\alpha\in H^{\frac{1}{4}%
}(0,T)$ allows to define a strongly continuous dynamical system in
$L^{2}(\mathbb{R})$, while $\alpha\in H^{\frac{3}{4}}(0,T)$ leads to the
existence of strong solutions to%
\begin{equation}
\left\{
\begin{array}
[c]{l}%
\frac{d}{dt}\psi=-iH_{\alpha(t)}\psi\\
\psi_{t=0}=\psi_{0}\in D(H_{\alpha(0)})
\end{array}
\right.  . \label{Cauchy 1}%
\end{equation}
The same strategy has also been adopted to study the diffusion problem with
nonautonomous delta potentials in $\mathbb{R}^{3}$ \cite{Dell'Anto}, and the
quantum evolution problem for a 1D nonlinear model where the parameter
$\alpha$ is assigned as a function of the particle's state \cite{AdTe}.

The techniques used in these works can be adapted to the 1D confining case.
However, the lack of a simple explicit expression for the free propagator
kernel and the use of eigenfunction expansions, which replaces the Fourier
transform analysis in this setting, necessarily requires some additional
efforts to obtain a result. In the perspective of the stability and
controllability analysis, we restrict ourselfs to the case $\alpha\in
H^{1}(0,T)$ and: $H_{\alpha}=-\frac{d^{2}}{dx^{2}}+\alpha\delta$ with
Dirichlet conditions on the boundary of $I=\left[  -\pi,\pi\right]  $. Under
these assumptions, we prove that the evolution problem (\ref{Cauchy 1}) admits
a strongly differentiable time propagator preserving the regularity of the
initial state. This turns out to be a key point to study the controllability
of the dynamics.

The controllability of a quantym dynamics through an external field has
attracted an increasing attention due to possible applications in nuclear
magnetic resonance, laser spectroscopy, photochemestry and quantum
information. This problem has been considered for confining Schr\"{o}dinger
operators with regular potentials of the form: $H=-\Delta_{\Omega}%
^{D}+V(x)+u(t)W(x)$, where $\Delta_{\Omega}^{D}$ is the Dirichlet Laplacian in
the bounded domain $\Omega\subset\mathbb{R}^{m}$, while $u$ is used as a
control function. The particular setting: $m=1$, $V=0$ and $W(x)=x$,
corresponding to a quantum particle confined in a 1D box and moving under the
action of a time dependent uniform eletric field, has been considered in
\cite{Beau}. The exact controllability of the quantum state was proved, in
this case, in $H^{7}$-neighbourhoods of the steady states by using $u\in
L^{2}(0,T)$ controls with $T>T_{m}>0$ (actually a simpler version of this
proof holding for all $T>0$ is given in \cite{BeLa}). For the same system, the
exact controllability between neighbourhoods of any couple of eigenstates is
discussed in \cite{BeCo}. In the more general setting, with $V,W\in C^{\infty
}\left(  \bar{\Omega},\mathbb{R}\right)  $ and $m$ being any space dimension,
differents approximate controllability results in $L^{2}$ have been presented
in \cite{ChMaSiBo} and \cite{Ner}. An example concerned with singular
pointwise potentials is presented in \cite{AdBo}, where point interaction
Hamiltonians are used to construct a general scheme allowing to steer the
system between the eigenstates of a drift operator $H_{0}$ (the 1D Dirichlet
Laplacian). The idea is to introduce adiabatic perturbations of the spectrum
of $H_{0}$ producing eigenvalues intersection, while controlling the state's
evolution through the adiabatic theorem. The result is the approximate state
to state controllability in infinite time. The particular case of 1D point
interactions is also related to control problems on quantum graphs where these
Hamiltonians naturally arise.

In Section \ref{Sec_control}, the nonlinear map $\alpha\rightarrow\psi_{T}$,
associating to the function $\alpha$ the solution at time $T$ of
(\ref{Cauchy 1}), is considered for a fixed initial state $\psi_{0}\in
D(H_{\alpha(0)})$. This can be regarded as a control system where the control
parameter is $\alpha$ and the target state is $\psi_{T}$. Using the regularity
properties of the time propagator, we show that: for $\psi_{0}\in H^{2}\cap
H_{0}^{1}(I)$, the map $\alpha\rightarrow\psi_{T}$ is of class $C^{1}%
(H_{0}^{1}(0,T),\,H^{2}\cap H_{0}^{1}(I))$. Then, the local controllability
relies on the surjectivity of the corresponding linearized map according to an
inverse function argument. This point is considered in Section
\ref{Sec_control_lin} where a general condition for the solution of the
linearized control problem is given in (\ref{Lin1})-(\ref{Lin2}) and
proposition \ref{Prop_moment}. When $\psi_{0}$ coincides with an even
eigenstate of the Diriclet Laplacian, this scheme provides with a result of
local steady state controllability in finite time.

\section{The Model}

Point interactions form a particular class of singular perturbations of the
Laplacian supported by finite set of points. In $\mathbb{R}^{d}$, $d\leq3$,
these Hamiltonians have been rigorously defined using the theory of
selfadjoint extensions of symmetric operators (e.g. in \cite{Albeverio}). The
definition easily extends to the case of bounded regions by taking into
account the Dirichlet conditions on the boundary for the functions in the
operator's domain (e.g. in \cite{A.M.}). In what follows we consider a delta
perturbation of the Dirichlet Laplacian, centered in the origin of the
interval $I=\left[  -\pi,\pi\right]  $. In terms of quadratic forms, this
family of Hamiltonians acts on $H_{0}^{1}(I)$ as%
\begin{equation}
H_{\alpha}=-\Delta_{I}^{D}+\alpha\,\delta\label{operatore 0}%
\end{equation}
where $\Delta_{I}^{D}$ is the Dirichlet Laplacian on the interval, $\delta$ is
the Dirac distribution centered in the origin, while $\alpha$ is a real
parameter. The operator's domain $D(H_{\alpha})$ extends to all those vectors
$\psi\in L^{2}(I)$, such that: $H_{\alpha}\psi\in L^{2}(I)$. The description
of $D(H_{\alpha})$ is strictly related to the properties of the Green's
kernel, $\mathcal{G}_{0}^{z}$, for the resolvent of the unperturbed operator
$\left(  -\Delta_{I}^{D}+z\right)  ^{-1}$. Fix $x^{\prime}\in I$; whenever $z$
does not belong to the spectrum of $\Delta_{I}^{D}$, $\mathcal{G}_{0}^{z}$
satisfies the equation%
\begin{equation}
\left(  -\Delta_{I}^{D}+z\right)  \mathcal{G}_{0}^{z}(x,x^{\prime}%
)=\delta(x-x^{\prime}),\quad z\in\mathbb{C}\backslash\sigma_{\Delta_{I}^{D}}
\label{Green 0}%
\end{equation}
The solution to (\ref{Green 0}) can be represented as the sum of the 'free'
Green's function plus an additional term taking into account the conditions at
the boundary%
\begin{equation}
\mathcal{G}_{0}^{z}(x,x^{\prime})=\frac{e^{-\sqrt{z}\left\vert x-x^{\prime
}\right\vert }}{2\sqrt{z}}-h(x,x^{\prime},z) \label{Green 1}%
\end{equation}%
\begin{equation}
\left\{
\begin{array}
[c]{l}%
\left(  -\Delta_{I}^{D}+z\right)  h(\cdot,x^{\prime},z)=0\\
\left.  h(\cdot,x^{\prime},z)\right\vert _{x=\pm\pi}=\left.  \frac
{e^{-\sqrt{z}\left\vert \cdot-x^{\prime}\right\vert }}{2\sqrt{z}}\right\vert
_{x=\pm\pi}%
\end{array}
\right.  \label{h 1}%
\end{equation}
The solution of (\ref{h 1}) is%
\begin{equation}
\mathcal{G}_{0}^{z}(x,x^{\prime})=\frac{e^{-\sqrt{z}\left\vert x-x^{\prime
}\right\vert }}{2\sqrt{z}}-\frac{1}{\sqrt{z}}\frac{e^{-\pi\sqrt{z}}}%
{e^{2\pi\sqrt{z}}-e^{-2\pi\sqrt{z}}}\left[  e^{\sqrt{z}x}\sinh\left(  \sqrt
{z}(\pi+x^{\prime})\right)  +e^{-\sqrt{z}x}\sinh\left(  \sqrt{z}(\pi
-x^{\prime})\right)  \right]  \label{Green 1.1}%
\end{equation}
Another useful representation of $\mathcal{G}_{0}^{z}$ is in terms of Fourier
expansion w.r.t. the Laplacian's eigenfunctions. Denoting with $\psi_{k}$ and
$\lambda_{k}$ respectively the eigenfunctions and the eigenvalues of
$-\Delta_{I}^{D}$, we have%
\begin{equation}
\psi_{k}(x)=\left\{
\begin{array}
[c]{l}%
\frac{1}{\sqrt{\pi}}\,\sin\frac{k}{2}x\quad k\ even\\
\frac{1}{\sqrt{\pi}}\,\cos\frac{k}{2}x\quad k\ odd
\end{array}
\right.  ;\quad\lambda_{k}=\frac{k^{2}}{4};\quad k\in\mathbb{N}
\label{stati stazionari}%
\end{equation}
and%
\begin{equation}
\mathcal{G}_{0}^{z}(x,x^{\prime})=\sum_{k\in\mathbb{N}}\frac{1}{\lambda_{k}%
+z}\,\psi_{k}^{\ast}(x^{\prime})\,\psi_{k}(x),\quad z\notin\sigma_{\Delta
_{I}^{D}} \label{Green 2.0}%
\end{equation}

Using (\ref{operatore 0}) and (\ref{Green 0}), a straightforward calculation
shows that $H_{\alpha}\psi\in L^{2}(I)$ whenever $\psi$ has the form%
\begin{equation}
\psi=\phi+q\mathcal{G}_{0}^{z}(\cdot,0),\ \phi\in D(\Delta_{I}^{D}%
),\ -q=\alpha\,\psi(0) \label{b.c. 1}%
\end{equation}
This is a general characterization of $D(H_{\alpha})$. It can be derived by
using the von Neumann theory of selfadjoint extensions to identify $H_{\alpha
}$ as the extension of the symmetric operator $H_{0}$%
\begin{equation}
\left\{
\begin{array}
[c]{l}%
D(H_{0})=\left\{  \psi\in H^{2}\cap H_{0}^{1}(I)\,\left\vert \,\psi
(0)=0\right.  \right\} \\
H_{\alpha}\psi=-\frac{d^{2}}{dx^{2}}\psi
\end{array}
\right.  \label{H_0}%
\end{equation}
to those vectors $\psi\in D(H_{0}^{\ast})$ fulfilling the selfadjoint boundary
conditions (e.g. in \cite{Albeverio 2} and \cite{Reed})%
\begin{equation}
\left\{
\begin{array}
[c]{l}%
\psi(0^{+})-\psi(0^{-})=0\\
\psi^{\prime}(0^{+})-\psi^{\prime}(0^{-})=\alpha\,\psi(0)
\end{array}
\right.  \label{b.c. 2}%
\end{equation}
The following proposition is a rephrasing of the result in \cite{Albeverio} in
the bounded domain case.

\begin{proposition}
\label{Proposition 0}Let $\alpha\in\mathbb{R}$, $\lambda\in\mathbb{C}%
\backslash\sigma_{\Delta_{I}^{D}}$ and denote with $H_{\alpha}$ the family of
selfadjoint extensions of $H_{0}$ associated to the boundary condition
(\ref{b.c. 2}). The following representation holds:%
\begin{equation}
D(H_{\alpha})=\left\{  \psi\in H^{2}(I\backslash\left\{  0\right\}  )\cap
H_{0}^{1}(I)\ \left\vert \ \psi=\phi^{\lambda}+q\mathcal{G}_{0}^{\lambda
}\left(  \cdot,0\right)  ;\ \phi^{\lambda}\in H^{2}\cap H_{0}^{1}%
(I);\ -q=\alpha\psi(0)\right.  \right\}  \label{dominio1}%
\end{equation}%
\begin{equation}
H_{\alpha}\psi=-\frac{d^{2}}{dx^{2}}\phi^{\lambda}-\lambda q\mathcal{G}%
_{0}^{\lambda}(\cdot,0) \label{operatore}%
\end{equation}
Moreover, for $z\in\mathbb{C}\backslash\mathbb{R}$, $\varphi\in L^{2}(I)$, the
resolvent $\left(  H_{\alpha}+z\right)  ^{-1}$ writes as:%
\begin{equation}
\left(  H_{\alpha}+z\right)  ^{-1}\varphi=\left(  -\Delta_{I}^{D}+z\right)
^{-1}\varphi-\frac{\alpha}{1+\alpha\,\mathcal{G}_{0}^{z}(0)}\left(
-\Delta_{I}^{D}+z\right)  ^{-1}\varphi(0)\,\mathcal{G}_{0}^{z}(\cdot,0)
\label{Resolvent}%
\end{equation}

\end{proposition}

\begin{remark}
Although the representation: $\psi=\phi^{\lambda}+q\mathcal{G}_{0}^{\lambda
}\left(  \cdot,0\right)  $ may change with different choices of $\lambda$, the
operator $H_{\alpha}$ depends only on the value of the parameter $\alpha$
related to the interaction's strength.
\end{remark}

Next we assign $\alpha$ as a function of time and consider the non autonomous
system defined by $H_{\alpha(t)}$%
\begin{equation}
\left\{
\begin{array}
[c]{l}%
\medskip i\frac{d}{dt}\psi(x,t)=H_{\alpha(t)}\psi(x,t)\,,\\
\psi(x,0)=\psi_{0}(x)\in D(H_{\alpha(0)})\,.
\end{array}
\right.  \label{Schroedinger}%
\end{equation}
The mild solutions are%
\begin{equation}
\psi(\cdot,t)=e^{it\Delta_{I}^{D}}\psi_{0}+i\int_{0}^{t}q(s)\,e^{i\left(
t-s\right)  \Delta_{I}^{D}}\delta\,ds\,, \label{Schroedinger 2-1}%
\end{equation}
where $e^{it\Delta_{I}^{D}}$ is the time propagator associated to $-\Delta
_{I}^{D}$ and $-q(t)=\alpha(t)\psi(0,t)$ fixes the boundary condition of the
operator's domain. The action of the operator $e^{it\Delta_{I}^{D}}$ on
$L^{2}(I)$ is%
\begin{equation}
e^{it\Delta_{I}^{D}}f=\sum\limits_{k\in\mathbb{N}}\left(  \psi_{k},f\right)
_{L^{2}(I)}\,e^{-i\lambda_{k}t}\psi_{k}\quad\forall f\in L^{2}(I)
\label{propagatore}%
\end{equation}
where the scalar product in $L^{2}$ is defined according to: $\left(
u,v\right)  _{L^{2}(I)}=%
{\displaystyle\int_{-\pi}^{\pi}}
\bar{u}\,v$. This relation suggests to replace the distributional part at the
r.h.s. of (\ref{Schroedinger 2-1}) with%
\[
\frac{i}{\sqrt{\pi}}\sum\limits_{k\in\mathbb{N}}\int_{0}^{t}q(s)\,e^{-i\lambda
_{k}\left(  t-s\right)  }\,ds\,\psi_{k}%
\]
It follows%
\begin{equation}
\psi(\cdot,t)=e^{it\Delta_{I}^{D}}\psi_{0}+\frac{i}{\sqrt{\pi}}\sum
\limits_{\substack{k\in\mathbb{N}\\k\ odd}}\int_{0}^{t}q(s)\,e^{-i\lambda
_{k}\left(  t-s\right)  }\,ds\,\psi_{k} \label{Schroedinger 3}%
\end{equation}%
\begin{equation}
q(t)=-\alpha(t)\left[  e^{it\Delta_{I}^{D}}\psi_{0}(0)+\frac{i}{\pi}%
\sum\limits_{\substack{k\in\mathbb{N}\\k\ odd}}\int_{0}^{t}q(s)\,e^{-i\lambda
_{k}\left(  t-s\right)  }\,ds\right]  \label{carica 1}%
\end{equation}
The essential informations about the dynamics described in
(\ref{Schroedinger 3})-(\ref{carica 1}) are contained in the auxiliary
variable $q$, usually referred to as the \emph{charge} of the particle (e.g.
in \cite{Albeverio}). In what follows, the solution of the above problem is
considered under suitable regularity assumptions for $\alpha$. The equivalence
with the original Cauchy problem (\ref{Schroedinger}) is further addressed.
Finally, the dependence of this solution from $\alpha$, considered as a
control parameter, is investigated. Our result is exposed in the following theorem.

\begin{theorem}
\label{Th}Assume: $\alpha\in H^{1}(0,T)$, $\psi_{0}\in D(H_{\alpha(0)})$ and
let $\psi_{T}(\psi_{0},\alpha)$, $q(t,\psi_{0},\alpha)$ respectively denote
the solution at time $T$ and the charge as functions of the intial state and
of the parameter $\alpha$. The following properties hold.\medskip\newline$1)$
The system (\ref{Schroedinger 3}) - (\ref{carica 1}) admits an unique solution
$\psi_{t}\in C(0,T;\,H^{1}(I))\cap C^{1}(0,T;\,L^{2}(I))$ with: $\psi_{t}\in
D(H_{\alpha(t)})$ and $i\partial_{t}\psi(\cdot,t)=H_{\alpha(t)}\psi(\cdot,t)$
at each $t$.\medskip\newline$2)$ The map $\alpha\rightarrow\psi_{T}(\psi
_{0},\alpha)$ is $C^{1}(H^{1}(0,T),\,H_{0}^{1}(I))$ in the sense of
Fr\'{e}chet. Moreover, the 'regular part' of the solution at time $T$, defined
by: $\psi_{T}(\psi_{0},\alpha)-q(T,\psi_{0},\alpha)\mathcal{G}_{0}^{\lambda}$
and considered as a function of $\alpha$, is of class $C^{1}(H^{1}%
(0,T),\,H^{2}\cap H_{0}^{1}(I))$. If $\alpha\in H_{0}^{1}(0,T)$, then
$\psi_{T}(\psi_{0},\alpha)$ coincides with its regular part and $\alpha
\rightarrow\psi_{T}(\psi_{0},\alpha)\in C^{1}(H_{0}^{1}(0,T),\,H^{2}\cap
H_{0}^{1}(I))$.\medskip\newline$3)$ Let $\alpha\in H_{0}^{1}(0,T)$ and
$\mathcal{W}$ be the subspace of $H^{2}\cap H_{0}^{1}(I)$ generated by the
system: $\left\{  \psi_{k},\ k\text{ odd}\right\}  $. If $\psi_{0}%
\in\mathcal{W}$, then: $\alpha\rightarrow\psi_{T}(\psi_{0},\alpha)\in
C^{1}(H_{0}^{1}(0,T),\,\mathcal{W})$. If $\psi_{0}=\psi_{\bar{k}}$ for a fixed
$\bar{k}$ odd and $T\geq8\pi$, then: it exist an open neighbourhood $V\times
P\subseteq H_{0}^{1}(0,T)\times\mathcal{W}$ of the point $\left(  0,\psi
_{\bar{k}}\right)  $ such that $\left.  \alpha\rightarrow\psi_{T}(\psi
_{0},\alpha)\right\vert _{V}:V\rightarrow P$ is surjective.
\end{theorem}

The proof is developed in the Propositions \ref{Proposition 3},
\ref{Proposition 4}, \ref{Proposition 4.1} and in the Theorem
\ref{local controllability} of sections 3 and 4.

\emph{Notation:} The Sobolev spaces are denoted with $H^{\nu}$ or $H_{0}^{\nu
}$ in the case of Dirichlet boundary conditions. The notation '$\lesssim$' is
to be intended as '$\leq C$' where $C$ is a suitable positive constant.

\section{The existence of the dynamics}

In what follows the existence of solutions to (\ref{Schroedinger}) is
discussed under the assumption: $\alpha\in H^{1}(0,T)$. Our strategy consists
in using the mild approach described by (\ref{Schroedinger 3})-(\ref{carica 1}%
). It articulates in two steps: First we prove that (\ref{carica 1}) admits an
unique solution in $H^{1}(0,T)$, for any finite time $T$ and any $\alpha\in
H^{1}(0,T)$. This result is then used to obtain estimates for the solution of
(\ref{Schroedinger}).

\subsection{The charge equation}

Let consider the linear map $U$%
\begin{equation}
Uq=\sum_{k\in D}\int_{0}^{t}q(s)e^{-i\lambda_{k}\left(  t-s\right)  }ds,\quad
D=\left\{  k\in\mathbb{N},\ k\ odd\right\}  ,\ \lambda_{k}=\frac{k^{2}}{4}
\label{U}%
\end{equation}

\begin{remark}
\label{Rem_base}Depending on the value of $T$, some of the functions
$e^{-i\lambda_{k}t}$ may belongs to the standard basis $\left\{  e^{i\omega
nt};\,\omega=\frac{2\pi}{T}\right\}  _{n\in\mathbb{Z}}$ of the space
$L^{2}(0,T)$. In particular, for: $T=8\pi N$, $N\in\mathbb{N}$, it follows%
\begin{equation}
\left\{  e^{-i\lambda_{k}t}\right\}  _{k\in\mathbb{N}}\subset\left\{
e^{i\omega nt};\,\omega=\frac{2\pi}{T}\right\}  _{n\in\mathbb{Z}}
\label{aux_condition}%
\end{equation}
The (\ref{aux_condition}) will be often used as an auxiliary condition in the
forthcoming analysis.
\end{remark}

\begin{lemma}
\label{Lemma 1.1}Let $\mathcal{H}_{T}$ denotes the Hilbert subspace%
\begin{equation}
\mathcal{H}_{T}\mathcal{=}\left\{  q\in H^{1}(0,T),\ q(0)=0\right\}  \label{H}%
\end{equation}
equipped with the $H^{1}$-norm. The map $U$ (\ref{U}) is bounded in
$\mathcal{H}_{T}$ with the estimate%
\begin{equation}
\left\Vert Uq\right\Vert _{\mathcal{H}_{T}}\lesssim\left(  T+1\right)
\,\left\Vert q\right\Vert _{\mathcal{H}_{T}} \label{St_0}%
\end{equation}

\end{lemma}

\begin{proof}
We use the standard basis $\left\{  e^{i\omega nt};\,\omega=\frac{2\pi}%
{T}\right\}  _{n\in\mathbb{Z}}$ of $L^{2}(0,T)$. The Fourier coefficients of
$\int_{0}^{t}q(s)e^{-i\lambda_{k}\left(  t-s\right)  }ds$ w.r.t. the vectors
$\,e^{i\omega nt}$ write as%
\[
\int_{0}^{T}dt\int_{0}^{t}ds\,q(s)e^{-i\lambda_{k}\left(  t-s\right)
}e^{-i\omega nt}=\int_{0}^{T}ds\,q(s)e^{i\lambda_{k}s}\int_{s}^{T}%
dt\,e^{-i\left(  \lambda_{k}+\omega n\right)  t}\,.
\]
Assume that $T$ fulfills the condition (\ref{aux_condition}) and,
consequently, $\left\{  e^{-i\lambda_{k}t}\right\}  $ forms a subset of
$\left\{  e^{i\omega nt};\,\omega=\frac{2\pi}{T}\right\}  _{n\in\mathbb{Z}}$.
This choice allows to avoid small divisors, while it does not imply any loss
of generality. With this assumption, the above expression writes as%
\[
\int_{0}^{T}dt\int_{0}^{t}ds\,q(s)e^{-i\lambda_{k}\left(  t-s\right)
}e^{-i\omega nt}=\left\{
\begin{array}
[c]{l}%
\bigskip\int_{0}^{T}ds\,q(s)e^{i\lambda_{k}s}\left(  T-s\right)
,\quad\text{for }\lambda_{k}+\omega n=0\,,\\
\frac{i}{\lambda_{k}+\omega n}\left[  \int_{0}^{T}ds\,q(s)e^{i\lambda_{k}%
s}-\int_{0}^{T}ds\,q(s)e^{-i\omega ns}\right]  ,\quad\text{for }\lambda
_{k}+\omega n\neq0\,.
\end{array}
\right.
\]
and the expansion%
\begin{gather}
\int_{0}^{T}dt\int_{0}^{t}ds\,q(s)e^{-i\lambda_{k}\left(  t-s\right)
}e^{-i\omega nt}=\qquad\qquad\qquad\qquad\qquad\qquad\qquad\qquad\qquad
\qquad\qquad\nonumber\\
=\left(  Tq_{-Nk^{2}}-\left(  tq\right)  _{-Nk^{2}}\right)  e^{-i\lambda_{k}%
t}+\sum_{\substack{n\in\mathbb{Z}\\\lambda_{k}+\omega n\neq0}}\frac{i}%
{\lambda_{k}+\omega n}\left[  q_{-Nk^{2}}-q_{n}\right]  e^{i\omega nt}
\label{Fourier}%
\end{gather}
follows, $q_{n}$ denoting the $n$-th Fourier coefficient of $q$ (in
particular: $q_{-Nk^{2}}=\left(  q,e^{-i\lambda_{k}t}\right)  _{L^{2}(0,T)}$).
Let $S=\left\{  n\in\mathbb{Z}:\omega n\notin\left\{  -\lambda_{k}\right\}
_{k\in D}\text{ }\right\}  $; we have%
\[
\int_{0}^{T}\left(  te^{-i\lambda_{k}t}\right)  e^{-i\omega nt}=\frac
{iT}{\lambda_{k}+\omega n},\quad\forall\,n\in S\,.
\]
This allows to write%
\[
\sum_{\substack{n\in\mathbb{Z}\\\lambda_{k}+\omega n\neq0}}\frac{i}%
{\lambda_{k}+\omega n}q_{-Nk^{2}}e^{i\omega nt}=q_{-Nk^{2}}\Pi_{S}\left(
\frac{t}{T}e^{-i\lambda_{k}t}\right)  \,,
\]
where $\Pi_{S}$ is the projector over the subspace spanned by $\left\{
e^{i\omega nt}\right\}  _{n\in S}$. Using (\ref{Fourier}) and the above
relation, the Fourier expansion of $Uq$ writes as%
\begin{equation}
Uq=\sum_{k\in D}\left(  Tq_{-Nk^{2}}-\left(  tq\right)  _{-Nk^{2}}\right)
e^{-i\lambda_{k}t}+\sum_{k\in D}q_{-Nk^{2}}\Pi_{S}\left(  \frac{t}%
{T}e^{-i\lambda_{k}t}\right)  -\sum_{k\in D}\sum_{n\in S}\frac{iq_{n}}%
{\lambda_{k}+\omega n}\,e^{i\omega nt}\,. \label{L1 1.2}%
\end{equation}
Consider the contributions to $Uq$:\newline$i)$ The norm of the first term at
the r.h.s. of (\ref{L1 1.2}) is bounded by%
\begin{equation}
\left\Vert \sum_{k\in D}\left(  Tq_{-Nk^{2}}-\left(  tq\right)  _{-Nk^{2}%
}\right)  e^{-i\lambda_{k}t}\right\Vert _{L^{2}(0,T)}\leq2T\left\Vert
q\right\Vert _{L^{2}(0,T)}\,. \label{L1 1.3}%
\end{equation}
$ii)$ For the second term, one has%
\begin{gather}
\left\Vert \sum_{k\in D}q_{-Nk^{2}}\Pi_{S}\left(  \frac{t}{T}e^{-i\lambda
_{k}t}\right)  \right\Vert _{L^{2}(0,T)}=\qquad\qquad\qquad\qquad\qquad
\qquad\qquad\qquad\qquad\qquad\nonumber\\
=\left\Vert \Pi_{S}\sum_{k\in D}q_{-Nk^{2}}\left(  \frac{t}{T}e^{-i\lambda
_{k}t}\right)  \right\Vert _{L^{2}(0,T)}\leq\sum_{k\in D}\left\vert
q_{-Nk^{2}}\right\vert ^{2}\leq\left\Vert q\right\Vert _{L^{2}(0,T)}\,.
\label{L1 1.4}%
\end{gather}
$iii)$ The remaining term is a superposition of $L^{2}$-functions: $\sum_{n\in
S}\frac{iq_{n}}{\lambda_{k}+\omega n}\,e^{i\omega nt}$ parametrized by the
index $k\in D$. Its $n$-th Fourier coefficient is formally expressed by:
$\sum_{k\in D}\frac{iq_{n}}{\lambda_{k}+\omega n}$ and the $L^{2}$-norm is%
\begin{equation}
\left\Vert \sum_{k\in D}\sum_{n\in S}\frac{iq_{n}}{\lambda_{k}+\omega
n}\,e^{i\omega nt}\right\Vert _{L^{2}(0,T)}^{2}=\sum_{n\in S}\left\vert
\sum_{k\in D}\frac{q_{n}}{\lambda_{k}+\omega n}\right\vert ^{2}\,.
\label{L1 2.0}%
\end{equation}
To study the sum at the r.h.s, the relation%
\begin{equation}
\frac{1}{\pi}\sum_{k\in D}\frac{1}{\lambda_{k}+z}=\mathcal{G}_{0}^{z}(0,0)\,,
\label{Green_formula}%
\end{equation}
arising from (\ref{stati stazionari}) and (\ref{Green 2.0}), is used.
According to (\ref{Green 1.1}) and (\ref{Green_formula}), it follows%
\begin{equation}
\sum_{k\in D}\frac{1}{\lambda_{k}+\omega n}=\frac{1}{2\sqrt{\omega n}}%
-\frac{1}{\sqrt{\omega n}}\frac{e^{-\pi\sqrt{\omega n}}}{e^{2\pi\sqrt{\omega
n}}-e^{-2\pi\sqrt{\omega n}}}2\sinh\left(  \pi\sqrt{\omega n}\right)
\,,\qquad n\in S\,. \label{Green 2.1}%
\end{equation}
This leads to%
\begin{equation}
\sum_{k\in D}\frac{1}{\lambda_{k}+\omega n}\underset{\left\vert n\right\vert
\rightarrow\infty}{\sim}\left\{
\begin{array}
[c]{ll}%
\medskip\mathcal{O}\left(  \frac{1}{\sqrt{\omega n}}e^{-\sqrt{\omega n}\pi
}\right)  & \quad\text{for }n>0\text{, }n\in S\\
\mathcal{O}\left(  \frac{1}{\sqrt{\left\vert \omega n\right\vert }}\right)  &
\quad\text{for }n<0\text{, }n\in S
\end{array}
\right.  \,. \label{Green 2.2}%
\end{equation}
Consider the map $T:f_{n}\rightarrow f_{n}\left(  \sum_{k\in D}\frac
{i}{\lambda_{k}+\omega n}\right)  $ in $\ell_{2}(S)$. Using (\ref{Green 2.2}),
$T$ can be identified with the limit of the sequence of finite rank maps:
$\left.  T_{N}f_{n}=\left\{  f_{n}\left(  \sum_{k\in D}\frac{i}{\lambda
_{k}+\omega n}\right)  \right\}  _{n\leq N}\right.  $, \ in the $\ell_{2}%
(S)$-operator norm. Thus, it defines a compact operator in $\ell_{2}(S)$ (e.g.
in Theorem VI.12, \cite{Simon1}), and we get%
\begin{equation}
\sum_{n\in S}\left\vert \sum_{k\in D}\frac{q_{n}}{\lambda_{k}+\omega
n}\right\vert ^{2}=\left\Vert \left(  \sum_{k\in D}\frac{q_{n}}{\lambda
_{k}+\omega n}\right)  \right\Vert _{\ell_{2}\left(  S\right)  }%
^{2}=\left\Vert T\left(  q_{n}\right)  \right\Vert _{\ell_{2}\left(  S\right)
}^{2}\lesssim\left\Vert q_{n}\right\Vert _{\ell_{2}\left(  S\right)  }^{2}\,,
\label{L1 2.1}%
\end{equation}%
\begin{equation}
\left\Vert \sum_{k\in D}\sum_{n\in S}\frac{iq_{n}}{\lambda_{k}+\omega
n}\,e^{i\omega nt}\right\Vert _{L^{2}(0,T)}\lesssim\left\Vert q\right\Vert
_{L^{2}(0,T)}\,. \label{L2 2.1}%
\end{equation}

The estimates (\ref{L1 1.3}), (\ref{L1 1.4}) and (\ref{L2 2.1}) yield%
\begin{equation}
\left\Vert Uq\right\Vert _{L^{2}(0,T)}\lesssim\left(  T+1\right)  \,\left\Vert
q\right\Vert _{L^{2}(0,T)} \label{L1 3}%
\end{equation}
This can be obviously extended to any finite time $\bar{T}\in\left(
0,+\infty\right)  $ by fixing $N_{\bar{T}}\in\mathbb{N}$ such that: $8\pi
N_{\bar{T}}\geq\bar{T}$ and setting $T=8\pi N_{\bar{T}}$ in
(\ref{aux_condition}).

Next, consider $q\in\mathcal{H}_{T}$. An integration by part of (\ref{U})
gives%
\begin{equation}
Uq=\sum_{k\in D}\frac{1}{i\lambda_{k}}q(t)-\sum_{k\in D}\frac{1}{i\lambda_{k}%
}\int_{0}^{t}q^{\prime}(s)e^{-i\lambda_{k}\left(  t-s\right)  }ds\,.
\label{C2-1}%
\end{equation}
Due to the asymptotic behaviour of the coefficients $\frac{1}{i\lambda_{k}%
}\sim\frac{1}{k^{2}}$, this sum uniformly converges to a continuous function
of $t\in\left[  0,T\right]  $, with
\begin{equation}
\frac{d}{dt}Uq=\sum_{k\in D}\int_{0}^{t}\dot{q}(s)e^{-i\lambda_{k}\left(
t-s\right)  }ds\,. \label{C1 1}%
\end{equation}
From (\ref{L1 3}), we get%
\begin{equation}
\left\Vert \frac{d}{dt}Uq\right\Vert _{L^{2}(0,T)}\lesssim\left(  T+1\right)
\left\Vert \dot{q}\right\Vert _{L^{2}(0,T)} \label{St_1_dUf}%
\end{equation}
and%
\begin{equation}
\left\Vert Uq\right\Vert _{H^{1}(0,T)}\lesssim\left(  T+1\right)  \,\left\Vert
q\right\Vert _{H^{1}(0,T)}\,. \label{St_2_dUf}%
\end{equation}

\end{proof}

In the next Lemma we give an iteration scheme for the solution of charge-like equations.

\begin{lemma}
\label{Lemma 1.3}Let $f,\varphi\in H^{1}(0,T)$, $\varphi$ real valued, and
$\mathcal{G}_{0}^{\lambda}(\cdot,0)$ be the Green' function defined in
(\ref{Green 1})-(\ref{h 1}) with $\lambda\in\mathbb{C}\backslash\mathbb{R}$.
The equation%
\begin{equation}
v=f-\varphi\left(  v(0)\left.  e^{it\Delta_{I}^{D}}\mathcal{G}_{0}^{\lambda
}(\cdot,0)\right\vert _{x=0}+\frac{i}{\pi}Uv\right)  \label{charge_scheme}%
\end{equation}
has a unique solution in $H^{1}(0,T)$ allowing the estimate%
\begin{equation}
\left\Vert v\right\Vert _{H^{1}(0,T)}\lesssim\left\Vert f\right\Vert
_{H^{1}(0,T)}\left(  1+\left\Vert \varphi\right\Vert _{H^{1}(0,T)}%
+\mathcal{P}\left(  \left\Vert \varphi\right\Vert _{H^{1}(0,T)}\right)
\right)  \label{charge_scheme_est}%
\end{equation}
where $\mathcal{P}(\cdot)$ is a suitable positive polynomial.
\end{lemma}

\begin{proof}
Since Lemma \ref{Lemma 1.1} provides with an estimate of $Uf$ in the
$\mathcal{H}_{T}$-norm, we introduce the variable $w(t)=v(t)-v(0)$. In this
setting, our problem writes as%
\begin{equation}
w=R-\varphi\frac{i}{\pi}Uw\,, \label{charge_resc}%
\end{equation}%
\begin{equation}
R(t)=f(t)-v(0)-\varphi(t)\left(  v(0)\left.  e^{it\Delta_{I}^{D}}%
\mathcal{G}_{0}^{\lambda}(\cdot,0)\right\vert _{x=0}+\frac{i}{\pi}U\left[
v(0)\right]  (t)\right)  \,. \label{R_0}%
\end{equation}
Next, the regularity of $R(t)$ is considered. Using the definition of $U$, and
the explicit expansion%
\begin{equation}
e^{it\Delta_{I}^{D}}\mathcal{G}_{0}^{\lambda}(0,0)=\frac{1}{\pi}\sum_{k\in
D}\frac{e^{-i\lambda_{k}t}}{\lambda_{k}+\lambda}\,, \label{source_green}%
\end{equation}
$R(t)$ writes as%
\begin{align}
R(t)  &  =f(t)-v(0)-v(0)\varphi(t)\left(  \left.  e^{it\Delta_{I}^{D}%
}\mathcal{G}_{0}^{\lambda}(\cdot,0)\right\vert _{x=0}+\frac{1}{\pi}\sum_{k\in
D}\frac{1}{\lambda_{k}}\left(  1-e^{-i\lambda_{k}t}\right)  \right)
\nonumber\\
&  =f(t)-v(0)-v(0)\varphi(t)\frac{1}{\pi}\sum_{k\in D}\left(  \frac{1}%
{\lambda_{k}}-e^{-i\lambda_{k}t}\frac{\lambda}{\lambda_{k}\left(  \lambda
_{k}+\lambda\right)  }\right)  \,,
\end{align}
where the last term at the r.h.s. is $H^{1}(\mathbb{R})$ for any
$-\lambda\notin\sigma_{-\Delta_{I}^{D}}$. Recalling that $H^{1}(0,T)$ is an
algebra w.r.t. the product, the above relation and $f,\varphi\in H^{1}(0,T)$
imply: $R\in H^{1}(0,T)$. Taking into account the initial condition $R(0)=0$
(following from the definition of the rescaled variable and the operator $U$
in (\ref{charge_resc})), we get: $R\in\mathcal{H}_{T}$ (see (\ref{H})).

Next consider the term $\varphi\frac{i}{\pi}Uw$ in (\ref{charge_resc}). For
$w\in\mathcal{H}_{T}$ and $\varphi\in H^{1}(0,T)$, the Lemma \ref{Lemma 1.1}
yields%
\begin{equation}
\left\Vert \varphi\frac{i}{\pi}Uw\right\Vert _{\mathcal{H}_{t}}\leq
C\frac{\left(  1+t\right)  \,}{\pi}\left\Vert \varphi\right\Vert _{H^{1}%
(0,t)}\left\Vert w\right\Vert _{\mathcal{H}_{t}} \label{contraction_est}%
\end{equation}
for any $t\in\left[  0,T\right]  $ and for a suitable $C>0$. To construct a
global solution, we notice that, for fixed $T,C>0$, $\varphi\in H^{1}(0,T)$,
there exist $\delta>0$ and a finite set $\left\{  t_{j}\right\}  _{j=1}^{N}$,
$t_{0}=0$ and $t_{j}>t_{j-1}$, fulfilling the conditions%
\begin{equation}
\left[  0,T\right]  =\cup_{j=1}^{N}\left[  t_{j-1},t_{j}\right]  \,,\qquad
t_{j}-t_{j-1}<\delta\,, \label{partition_con}%
\end{equation}%
\begin{equation}
C\left(  1+\delta\right)  \,\left\Vert \varphi\right\Vert _{H^{1}%
(t_{j-1},t_{j})}<\pi\,c_{\delta}\qquad\text{for all }j\,,
\label{contraction_con}%
\end{equation}
with $c_{\delta}<1$, depending on $\delta$. According to
(\ref{contraction_est}) and (\ref{contraction_con}) with $j=1$, $\varphi
\frac{i}{\pi}U$ is contraction map in $\mathcal{H}_{t_{1}}$ and the equation
(\ref{charge_resc})-(\ref{R_0}) admits an unique solution, $w_{1}=1_{\left[
0,t_{1}\right]  }w$, bounded by%
\[
\left\Vert w_{1}\right\Vert _{\mathcal{H}_{t_{1}}}\lesssim\left\Vert
R\right\Vert _{\mathcal{H}_{t_{1}}}\lesssim\left\Vert f\right\Vert
_{H^{1}(0,t_{1})}+\left\vert v(0)\right\vert \left(  1+\left\Vert
\varphi\right\Vert _{H^{1}(0,t_{1})}\right)
\]
Using (\ref{charge_scheme}), the initial value of $v$ if formally related to
$f(0)$ by: $v(0)\left(  1+\varphi(0)\mathcal{G}_{0}^{\lambda}(0,0)\right)
=f(0)$; this allows to define $v(0)$ privided that $\left(  1+\varphi
(0)\mathcal{G}_{0}^{\lambda}(0,0)\right)  $ is not null. According to the
resolvent's formula (\ref{Resolvent}), the condition: $\left(  1+\varphi
(0)\mathcal{G}_{0}^{z}(0,0)\right)  =0$, with $z\notin\sigma_{-\Delta_{I}^{D}%
}$, corresponds to the eigenvalue equation for the Hamiltonian $H_{\varphi}$.
Since this is a selfadjoint model, the assumption $\lambda\in\mathbb{C}%
\backslash\mathbb{R}$ implies: $\left(  1+\varphi(0)\mathcal{G}_{0}^{\lambda
}(0,0)\right)  \neq0$. It follows that: $v(0)=\left(  1+\varphi(0)\mathcal{G}%
_{0}^{\lambda}(0,0)\right)  ^{-1}f(0)$ and $\left\vert v(0)\right\vert
\lesssim\left\vert f(0)\right\vert \,.$This leads to%
\[
\left\Vert w_{1}\right\Vert _{\mathcal{H}_{t_{1}}}\lesssim\left\Vert
f\right\Vert _{H^{1}(0,t_{1})}\left(  1+\left\Vert \varphi\right\Vert
_{H^{1}(0,t_{1})}\right)  \,.
\]
The definition: $w_{1}=1_{\left[  0,t_{1}\right]  }\left(  v-v(0)\right)  $,
provides with a similar estimate for the function $v$%
\begin{equation}
\left\Vert v\right\Vert _{H^{1}(0,t_{1})}\lesssim\left\Vert f\right\Vert
_{H^{1}(0,t_{1})}\left(  1+\left\Vert \varphi\right\Vert _{H^{1}(0,t_{1}%
)}\right)  \,. \label{charge_bound3}%
\end{equation}

This can be extended to larger times by the following iteration scheme:%
\begin{equation}
w_{j+1}(t)=v(t)-v(t_{j})\,,\qquad t\in\left[  t_{j},t_{j+1}\right]  \,,\qquad
j=1,...,N-1 \label{charge_iteration}%
\end{equation}%
\begin{equation}
w_{j+1}(t^{\prime}+t_{j})=R_{j+1}(t^{\prime}+t_{j})-\varphi(t^{\prime}%
+t_{j})\frac{i}{\pi}\left[  Uw_{j+1}(\cdot+t_{j})\right]  (t^{\prime
})\,,\qquad t^{\prime}\in\left[  0,t_{j+1}-t_{j}\right]
\label{charge_iteration_eq}%
\end{equation}%
\begin{align}
R_{j+1}(t)  &  =f(t)-v(t_{j})-\varphi(t)\left(  v(0)\left.  e^{it\Delta
_{I}^{D}}\mathcal{G}_{0}^{\lambda}(\cdot,0)\right\vert _{x=0}\right)
\nonumber\\
&  -\frac{i}{\pi}\varphi(t)\left(  U\left[  v(t_{j})\right]  +\sum_{k\in
D}\int_{0}^{t_{j}}\left(  v(s)-v(t_{j})\right)  e^{-i\lambda_{k}\left(
t-s\right)  }ds\right)  \label{source_j+1}%
\end{align}
with the source term $R_{j+1}$ depending at each step on the past solution
$1_{\left[  0,t_{j}\right]  }q$. From this definition, it follows:
$R_{j+1}(t_{j})=0$, while an integration by part gives%
\begin{equation}
R_{j+1}(t)=f(t)-v(t_{j})-\frac{\varphi(t)}{\pi}\sum_{k\in D}\left(
\frac{v(t_{j})}{\lambda_{k}}+v(0)\left(  \frac{-\lambda e^{-i\lambda_{k}t}%
}{\lambda_{k}\left(  \lambda_{k}+\lambda\right)  }\right)  -\frac{1}%
{\lambda_{k}}\int_{0}^{t_{j}}\dot{v}(s)e^{-i\lambda_{k}\left(  t-s\right)
}ds\right)  \label{source_j+1_exp}%
\end{equation}
where (\ref{source_green}) have been taken into account. If $v\in
H^{1}(0,t_{j})$, all the contributions at the r.h.s. of (\ref{source_j+1_exp})
are $H^{1}(0,T)$ and $R_{j+1}\in\mathcal{H}_{t_{j+1}-t_{j}}$. Then, the
contractivity property of the operator, expressed by
(\ref{charge_iteration_eq}), implies the existence of an unique solution
$w_{j+1}\in\mathcal{H}_{t_{j+1}-t_{j}}$ with the bound%
\begin{equation}
\left\Vert w_{j+1}\right\Vert _{\mathcal{H}_{t_{j+1}-t_{j}}}\lesssim\left\Vert
R_{j+1}\right\Vert _{\mathcal{H}_{t_{j+1}-t_{j}}}\lesssim\left\Vert
f\right\Vert _{H^{1}(t_{j},t_{j+1})}+\left\Vert v\right\Vert _{H^{1}(0,t_{j}%
)}+\left\Vert \varphi\right\Vert _{H^{1}(0,T)}\left\Vert v\right\Vert
_{H^{1}(0,t_{j})}\,. \label{charge_bound1}%
\end{equation}
Starting from (\ref{charge_bound3}), an induction argument based on the
itarated use of (\ref{charge_bound1}) leads to%
\begin{equation}
\left\Vert v\right\Vert _{H^{1}(0,T)}\lesssim\left\Vert f\right\Vert
_{H^{1}(0,T)}\left(  1+\left\Vert \varphi\right\Vert _{H^{1}(0,T)}%
+\mathcal{P}\left(  \left\Vert \varphi\right\Vert _{H^{1}(0,T)}\right)
\right)
\end{equation}
where $\mathcal{P}(\cdot)$ is a positive polynomial.
\end{proof}

A first application of this Lemma gives an $H^{1}$-bound for the solutions to
the charge equation.

\begin{corollary}
\label{Coroll_1.3_1}Let $\alpha\in H^{1}(0,T;\mathbb{R})$ and $\psi_{0}\in
D(H_{\alpha(0)})$ and denote with $\phi_{0}^{\lambda}$ the regular part of
$\psi_{0}$ defined according to the representation (\ref{dominio1}).
Then:\medskip\newline$i)$ The equation (\ref{carica 1}) admits an unique
solution $q\in H^{1}(0,T)$ such that%
\begin{equation}
\left\Vert q\right\Vert _{H^{1}(0,T)}\leq C_{\lambda,\phi_{0}^{\lambda}%
}\left\Vert \alpha\right\Vert _{H^{1}(0,T)}\left(  1+\mathcal{P}\left(
\left\Vert \alpha\right\Vert _{H^{1}(0,T)}\right)  \right)
\label{charge_bound}%
\end{equation}
where $\mathcal{P}\left(  \cdot\right)  $ is a positive polynomial,
$C_{\lambda,\phi_{0}^{\lambda}}$ is a positive constant depending on
$\lambda\in\mathbb{C}\backslash\mathbb{R}$ and $\phi_{0}^{\lambda}$%
.\medskip\newline$ii)$ For a fixed $\psi_{0}\in D(H_{\alpha(0)})$, the map
$\alpha\rightarrow q$ is locally Lipschitzian in $H^{1}(0,T)$.
\end{corollary}

\begin{proof}
$i)$ With the notation introduced in (\ref{U}), the charge equation is%
\begin{equation}
q=-\alpha e^{it\Delta_{I}^{D}}\psi_{0}(0)-\alpha\frac{i}{\pi}Uq\,.
\label{charge_eq}%
\end{equation}
Using the decomposition: $\psi_{0}=\phi_{0}^{\lambda}+q(0)\mathcal{G}%
_{0}^{\lambda}(\cdot,0)$, $\phi_{0}^{\lambda}\in H^{2}\cap H_{0}^{1}(I)$,
holding for $\lambda\in\mathbb{C}\backslash\mathbb{R}$ (we refer to
Proposition \ref{Proposition 0}), this equation rephrases as follows%
\begin{equation}
q=-\alpha e^{it\Delta_{I}^{D}}\phi_{0}^{\lambda}-\alpha\left(  q(0)e^{it\Delta
_{I}^{D}}\left.  \mathcal{G}_{0}^{\lambda}(\cdot,0)\right\vert _{x=0}+\frac
{i}{\pi}Uq\right)  \,.
\end{equation}
Consider $e^{it\Delta_{I}^{D}}\phi_{0}^{\lambda}(0)$. As a consequence of the
Stone's Theorem (e.g. in \cite{Simon})), the operator $e^{it\Delta_{I}^{D}}$
defines a continuous flow on $H^{2}\cap H_{0}^{1}(I)$ strongly differentiable
in $\,L^{2}(I)$. For $\left.  \phi_{0}^{\lambda}\in H^{2}\cap H_{0}%
^{1}(I)\right.  $, we use the Fourier expansion $\phi_{0}^{\lambda}%
=\sum\limits_{k\in\mathbb{N}}a_{k}\psi_{k}$, where the coefficients
$a_{k}=\left(  \phi_{0}^{\lambda},\psi_{k}\right)  _{L^{2}(I)}$ are
characterized by: $a_{k}\in\ell_{2}(\mathbb{N})$, $\lambda_{k}a_{k}\in\ell
_{2}(\mathbb{N})$ (as it follows integrating by parts and exploiting the
boundary conditions of $\psi_{k}$). The time propagator $e^{it\Delta_{I}^{D}}$
acts like: $e^{it\Delta_{I}^{D}}\phi_{0}^{\lambda}=\sum\limits_{k\in
\mathbb{N}}a_{k}e^{-i\lambda_{k}t}\psi_{k}$. In particular, we have%
\begin{equation}
e^{it\Delta_{I}^{D}}\phi_{0}^{\lambda}(0)=\sum\limits_{k\in D}a_{k}%
e^{-i\lambda_{k}t}\,. \label{termine noto}%
\end{equation}
According to: $\lambda_{k}a_{k}\in\ell_{2}$, it follows: $\partial
_{t}e^{it\Delta_{I}^{D}}\phi_{0}^{\lambda}(0)=-i\sum_{k\in D}a_{k}\lambda
_{k}e^{-i\lambda_{k}t}$, in the weak sense and: $e^{it\Delta_{I}^{D}}\phi
_{0}^{\lambda}(0)\in H^{1}(0,T)$. Then the first statement follows as an
application of Lemma \ref{Lemma 1.3} with $f=-\alpha e^{it\Delta_{I}^{D}}%
\phi_{0}^{\lambda}(0)$ and $\varphi=\alpha$.

$ii)$ Let $\alpha$, $\tilde{\alpha}\in H^{1}(0,T)$, $\psi_{0}$ and
$\tilde{\psi}_{0}$ defined with the same regular part%
\begin{equation}
\psi_{0}=\phi_{0}^{\lambda}+q(0)\mathcal{G}_{0}^{\lambda}(\cdot,0)\,,\qquad
\tilde{\psi}_{0}=\phi_{0}^{\lambda}+\tilde{q}(0)\mathcal{G}_{0}^{\lambda
}(\cdot,0)\,, \label{charge_cont_psi}%
\end{equation}
and consider the corresponding solutions to (\ref{charge_eq}) $q$ and
$\tilde{q}$. The initial values depends on $\alpha$ and $\tilde{\alpha}$
through the conditions%
\begin{equation}
q(0)=\frac{\phi_{0}^{\lambda}(0)}{1+\alpha(0)\mathcal{G}_{0}^{\lambda}%
(0,0)}\,,\qquad\tilde{q}(0)=\frac{\phi_{0}^{\lambda}(0)}{1+\tilde{\alpha
}(0)\mathcal{G}_{0}^{\lambda}(0,0)}\,, \label{charge_cont_init}%
\end{equation}
while the difference $u=q-\tilde{q}$\ solves the equation%
\begin{equation}
u=S-\alpha\left(  u(0)\,e^{it\Delta_{I}^{D}}\left.  \mathcal{G}_{0}^{\lambda
}(\cdot,0)\right\vert _{x=0}+\frac{i}{\pi}Uu\right)  \,,
\label{charge_diff_eq}%
\end{equation}%
\begin{equation}
S(t)=-\left(  \alpha-\tilde{\alpha}\right)  \left[  e^{it\Delta_{I}^{D}}%
\phi_{0}^{\lambda}(0)-\frac{\tilde{q}}{\tilde{\alpha}}\right]  \,.
\label{charge_diff_source}%
\end{equation}
From the previous point, we have: $e^{it\Delta_{I}^{D}}\phi_{0}^{\lambda
}(0),q,\tilde{q}\in H^{1}(0,T)$. The same inclusion holds for the ratios
$\frac{q}{\alpha},\frac{\tilde{q}}{\tilde{\alpha}}$, according to the
equation's structure. An application of Lemma \ref{Lemma 1.3} with: $f=S$ and
$\varphi=\alpha$, gives%
\begin{equation}
\left\Vert u\right\Vert _{H^{1}(0,T)}\lesssim\left\Vert \alpha-\tilde{\alpha
}\right\Vert _{H^{1}(0,T)} \label{charge_diff_est}%
\end{equation}

\end{proof}

\subsection{Solution of the evolution problem}

Next we consider the system (\ref{Schroedinger 3})-(\ref{carica 1}) with%
\begin{equation}
\alpha\in H^{1}(0,T),\quad\psi_{0}\in D(H_{\alpha(0)})\,.
\label{stato iniziale}%
\end{equation}
Due to the result of Lemma \ref{Lemma 1.3}, the equation (\ref{carica 1}) has
a unique solution $q\in H^{1}(0,T)$. This can be used to show that the
corresponding evolution, defined by (\ref{Schroedinger 3}), is $C\left(
\left[  0,T\right]  ,H_{0}^{1}(I)\right)  \cap C^{1}(\left[  0,T\right]
,L^{2}(I))$ and solves the Cauchy problem (\ref{Schroedinger}). Fix
$\lambda\in\mathbb{C}\backslash\sigma_{\Delta_{I}^{D}}$; from the definition
of the operator's domain (see (\ref{dominio1})), any $\psi\in D(H_{\alpha
(t)})$ can be represented as the sum of a 'regular' part plus a Green's
function%
\begin{equation}
\psi(\cdot,t)=\phi^{\lambda}(\cdot,t)+q(t)\mathcal{G}_{0}^{\lambda}
\label{state}%
\end{equation}
with $q$ fulfilling (\ref{carica 1}). In the mild sense, the evolution of the
regular part is%
\begin{equation}
\phi^{\lambda}(x,t)=e^{it\Delta_{I}^{D}}\psi_{0}(x)+\frac{i}{\sqrt{\pi}}%
\sum\limits_{k\in D}\int_{0}^{t}q(s)e^{-i\lambda_{k}\left(  t-s\right)
}\,ds\,\psi_{k}(x)-q(t)\mathcal{G}_{0}^{\lambda}(x)\,. \label{phi 1}%
\end{equation}
Next we consider the operator%
\begin{equation}
F(q,t)=\frac{i}{\sqrt{\pi}}\sum\limits_{k\in D}\left(  \int_{0}^{t}%
q(s)e^{-i\lambda_{k}\left(  t-s\right)  }\,ds\right)  \,\psi_{k}\,. \label{F}%
\end{equation}

\begin{lemma}
\label{Lemma 1.2}For $q\in H^{1}(0,T)$, the map $F(q,t)$ is bounded in
$C\left(  \left[  0,T\right]  ,\,H_{0}^{1}(I)\right)  $.
\end{lemma}

\begin{proof}
Let $q_{t}=1_{\left(  0,t\right)  }q$ and $N_{T}$ be the smallest integer such
that: $8\pi N_{T}\geq T$. By definition (see (\ref{aux_condition})), $\left\{
e^{-i\lambda_{k}s}\right\}  _{k\in\mathbb{N}}$ forms a subset of the standard
basis in $L^{2}(0,8\pi N_{T})$. The Fourier coefficients of $q_{t}$ along
these frequencies are%
\begin{equation}
\int_{0}^{8\pi N_{T}}q_{t}(s)e^{i\lambda_{k}s}\,ds=\int_{0}^{t}q(s)e^{i\lambda
_{k}s}\,ds=c_{k}(t) \label{L1.2 0}%
\end{equation}
and the inequality%
\begin{equation}
\sum_{k\in\mathbb{N}}\left\vert c_{k}(t)\right\vert ^{2}\leq\left\Vert
q_{t}\right\Vert _{L^{2}(0,8\pi N_{T})}^{2}=\int_{0}^{t}\left\vert
q(s)\right\vert ^{2}\,ds \label{L1.2 1}%
\end{equation}
holds. It follows%
\begin{equation}
\left\Vert F(q,t)\right\Vert _{L^{2}(I)}^{2}=\frac{1}{\pi}\sum_{k\in
D}\left\vert c_{k}(t)\right\vert ^{2}\leq\frac{1}{\pi}\int_{0}^{t}\left\vert
q(s)\right\vert ^{2}\,ds \label{L1.2 2}%
\end{equation}
and%
\begin{equation}
\sup_{t\in\left[  0,T\right]  }\left\Vert F(q,t)\right\Vert _{L^{2}(I)}%
^{2}\leq\frac{1}{\pi}\int_{0}^{T}\left\vert q(s)\right\vert ^{2}\,ds=\frac
{1}{\pi}\left\Vert q\right\Vert _{L^{2}(0,T)}^{2}\,. \label{L1.2 3}%
\end{equation}
For $q\in H^{1}(0,T)$, an integration by part gives%
\begin{equation}
F(q,t)=\frac{1}{\sqrt{\pi}}\sum\limits_{k\in D}\frac{1}{\lambda_{k}}\left(
q(t)-q(0)e^{-i\lambda_{k}t}-\int_{0}^{t}q^{\prime}(s)e^{-i\lambda_{k}\left(
t-s\right)  }\,ds\,\right)  \psi_{k}\,, \label{C1.2 0}%
\end{equation}
while, using the relation: $\frac{d}{dx}\psi_{k}=\lambda_{k}^{\frac{1}{2}}%
\psi_{k}$, the weak derivative $\frac{d}{dx}F(q,t)$ can be written as%
\begin{equation}
\frac{d}{dx}F(q,t)=\frac{1}{\sqrt{\pi}}\sum\limits_{k\in D}\frac{1}%
{\lambda_{k}^{\frac{1}{2}}}\left(  q(t)-q(0)e^{-i\lambda_{k}t}-\int_{0}%
^{t}q^{\prime}(s)e^{-i\lambda_{k}\left(  t-s\right)  }\,ds\,\right)  \psi
_{k}\,. \label{C1.2 1}%
\end{equation}
For the first term at the r.h.s. of (\ref{C1.2 1}), the asymptotic behaviour
of the coefficients $\lambda_{k}^{\frac{1}{2}}=\frac{k}{2}$ and the
assumption: $q\in H^{1}(0,T)$ implies: $\left.  \sum_{k\in D}\lambda
_{k}^{-\frac{1}{2}}\left(  q(t)-q(0)e^{-i\lambda_{k}t}\right)  \psi_{k}\in
C(\left(  \left[  0,T\right]  ,L^{2}(I)\right)  )\right.  $. For the remaining
term, we proceed as before: let $q_{t}^{\prime}=1_{\left(  0,t\right)
}q^{\prime}$ and denote with $c_{k}^{\prime}$ the Fourier coefficients of
$q_{t}^{\prime}$ along the freqencies $\left\{  e^{-i\lambda_{k}s}\right\}
_{k\in\mathbb{N}}$%
\[
c_{k}^{\prime}=\int_{0}^{t}q^{\prime}(s)e^{-i\lambda_{k}\left(  t-s\right)
}\,ds\,.
\]
It follows%
\begin{equation}
\left\Vert \sum\limits_{k\in D}\frac{1}{\lambda_{k}^{\frac{1}{2}}}\int_{0}%
^{t}q^{\prime}(s)e^{-i\lambda_{k}\left(  t-s\right)  }\,ds\,\psi
_{k}\right\Vert _{L^{2}(I)}^{2}=\sum_{k\in D}\frac{\left\vert c_{k}^{\prime
}\right\vert ^{2}}{\lambda_{k}}\leq\int_{0}^{t}\left\vert q^{\prime
}(s)\right\vert ^{2}\,ds \label{C1.2 2}%
\end{equation}
and%
\begin{equation}
\sup_{t\in\left[  0,T\right]  }\left\Vert \sum\limits_{k\in D}\frac{1}%
{\lambda_{k}^{\frac{1}{2}}}\int_{0}^{t}q^{\prime}(s)e^{-i\lambda_{k}\left(
t-s\right)  }\,ds\,\psi_{k}\right\Vert _{L^{2}(I)}^{2}\leq\left\Vert
q^{\prime}\right\Vert _{L^{2}(0,T)}^{2}\,. \label{C1.2 3}%
\end{equation}
The above estimates lead us to%
\begin{equation}
\sup_{t\in\left[  0,T\right]  }\left\Vert F(q,t)\right\Vert _{H_{0}^{1}%
(I)}\lesssim\left\Vert q\right\Vert _{H^{1}(0,T)}\,. \label{bound1.0}%
\end{equation}

\end{proof}

\begin{proposition}
\label{Proposition 3}Let $\alpha\in H^{1}(0,T)$ and $\psi_{0}\in
D(H_{\alpha(0)})$. The system (\ref{Schroedinger 3})-(\ref{carica 1}), has an
unique solution $\psi_{t}\in C\left(  \left[  0,T\right]  ,\,H_{0}%
^{1}(I)\right)  \cap C^{1}(\left[  0,T\right]  ,\,L^{2}(I))$ such that:
$\psi_{t}\in D(H_{\alpha(t)})$ and $i\partial_{t}\psi_{t}=H_{\alpha(t)}%
\psi_{t}$ at each $t$.
\end{proposition}

\begin{proof}
The proof articulates in two steps. We first consider the conditions $\psi
_{t}\in C\left(  \left[  0,T\right]  ,\,H_{0}^{1}(I)\right)  $ and $\psi
_{t}\in C^{1}(\left[  0,T\right]  ,\,L^{2}(I))$. Then we discuss the
equivalence of the system (\ref{Schroedinger 3})-(\ref{carica 1}) with the
initial problem.
\end{proof}

\begin{proof}
$1)$ Using the notation introduced in (\ref{F}), the solution $\psi_{t}$ of
(\ref{Schroedinger 3})-(\ref{carica 1}) writes as%
\begin{equation}
\psi_{t}=e^{it\Delta_{I}^{D}}\psi_{0}+F(q,t)
\end{equation}
with $q\in H^{1}(0,T)$ solving the charge equation (\ref{carica 1}), and
$\psi_{0}\in D(H_{\alpha(0)})$. Due to the domain's structure, $\psi_{0}\in
H_{0}^{1}(I)$: in this case, the term $e^{it\Delta_{I}^{D}}\psi_{0}$ defines a
$C(\left[  0,T\right]  ,\,H_{0}^{1}(I))$ map (as a consequence of the Stone's
theorem, e.g. in \cite{Simon}). Moreover, following the result of Lemma
\ref{Lemma 1.2}, one has: $F(q,t)\in C\left(  \left[  0,T\right]  ,\,H_{0}%
^{1}(I)\right)  $. In order to study the $C^{1}(\left[  0,T\right]
,\,L^{2}(I))$ regularity of $\psi_{t}$, one can use the decomposition stated
in Proposition \ref{Proposition 0}: $\psi_{0}=\phi_{0}^{\lambda}%
+q(0)\mathcal{G}_{0}^{\lambda}(\cdot,0)$, with $\phi_{0}^{\lambda}\in
H^{2}\cap H_{0}^{1}(I)$ and $\lambda\in\mathbb{C}\backslash\sigma_{\Delta
_{I}^{D}}$. For $w(t)=q(t)-q(0)$, $\psi_{t}$ writes as%
\begin{equation}
\psi_{t}=e^{it\Delta_{I}^{D}}\phi_{0}^{\lambda}+F(w,t)+Z(t)\,, \label{ref1}%
\end{equation}%
\[
Z(t)=q(0)\left[  e^{it\Delta_{I}^{D}}\mathcal{G}_{0}^{\lambda}(\cdot
,0)+F(1,t)\right]  \,.
\]
Due to the regularity of $\phi_{0}^{\lambda}$, one has: $e^{it\Delta_{I}^{D}%
}\phi_{0}^{\lambda}\in C^{1}(\left[  0,T\right]  ,\,L^{2}(I))$. Next consider
$F(w,t)$: an integration by part gives%
\[
F(w,t)=\frac{1}{\sqrt{\pi}}\sum\limits_{k\in D}\frac{1}{\lambda_{k}}\left(
w(t)-\int_{0}^{t}w^{\prime}(s)e^{-i\lambda_{k}\left(  t-s\right)
}\,ds\,\right)  \psi_{k}%
\]
and%
\begin{equation}
\frac{d}{dt}F(w,t)=\frac{i}{\sqrt{\pi}}\sum_{k\in D}\int_{0}^{t}w^{\prime
}(s)e^{-i\lambda_{k}\left(  t-s\right)  }ds\,\psi_{k}=F(w^{\prime},t)\,.
\label{ref2}%
\end{equation}
Using the estimate (\ref{L1.2 3}) in Lemma \ref{Lemma 1.2}, this relation
gives: $F(w,t)\in C^{1}(\left[  0,T\right]  ,\,L^{2}(I))$. Concerning the term
$Z(t)$, the formula%
\[
e^{it\Delta_{I}^{D}}\mathcal{G}_{0}^{\lambda}(\cdot,0)=\frac{1}{\sqrt{\pi}%
}\sum_{k\in D}\frac{e^{-i\lambda_{k}t}}{\lambda_{k}+\lambda}\psi_{k}%
\]
and an explicit computation lead to%
\begin{equation}
Z(t)=\frac{q(0)}{\sqrt{\pi}}\sum_{k\in D}\frac{1}{\lambda_{k}}\psi_{k}%
-\frac{q(0)}{\sqrt{\pi}}\sum_{k\in D}\frac{\lambda e^{-i\lambda_{k}t}}%
{\lambda_{k}\left(  \lambda_{k}+\lambda\right)  }\psi_{k}\,. \label{Zeta}%
\end{equation}
It follows: $Z(t)\in C^{1}(\left[  0,T\right]  ,\,H_{0}^{1}(I))$ and:
$\psi_{t}\in C\left(  \left[  0,T\right]  ,\,H_{0}^{1}(I)\right)  \cap
C^{1}(\left[  0,T\right]  ,\,L^{2}(I))$.

$2)$ The regular part of $\psi_{t}$ writes as%
\begin{equation}
\phi_{t}^{\lambda}=e^{it\Delta_{I}^{D}}\psi_{0}+F(q,t)-q(t)\mathcal{G}%
_{0}^{\lambda}(\cdot,0) \label{phi 0}%
\end{equation}
Using once more the decomposition: $\psi_{0}=\phi_{0}^{\lambda}%
+q(0)\mathcal{G}_{0}^{\lambda}(\cdot,0)$, $\phi_{0}^{\lambda}\in H^{2}\cap
H_{0}^{1}(I)$, $\lambda\in\mathbb{C}\backslash\sigma_{\Delta_{I}^{D}}$ and
$w(t)=q(t)-q(0)$, we get%
\begin{equation}
\phi_{t}^{\lambda}=e^{it\Delta_{I}^{D}}\phi_{0}^{\lambda}%
+F(w,t)-w(t)\mathcal{G}_{0}^{\lambda}(\cdot,0)+q(0)Q(t) \label{phi 0.1}%
\end{equation}%
\begin{equation}
Q(t)=\left[  \left(  e^{it\Delta_{I}^{D}}-1\right)  \mathcal{G}_{0}^{\lambda
}(\cdot,0)+F(1,t)\right]  \label{Q(t)}%
\end{equation}
The regularity of $\phi_{0}^{\lambda}$, yields: $e^{it\Delta_{I}^{D}}\phi
_{0}^{\lambda}\in C(\left[  0,T\right]  ,\,H^{2}\cap H_{0}^{1}(I))$, while,
according to the explicit form of $e^{it\Delta_{I}^{D}}\mathcal{G}%
_{0}^{\lambda}(\cdot,0)$ and $F(1,t)$, $Q(t)$ is%
\begin{equation}
Q(t)=\frac{1}{\sqrt{\pi}}\sum_{k\in D}\left(  1-e^{-i\lambda_{k}t}\right)
\frac{\lambda}{\lambda_{k}\left(  \lambda_{k}-\lambda\right)  }\psi_{k}
\label{Qu}%
\end{equation}
This yields $Q(t)\in C(\left[  0,T\right]  ,\,H^{2}\cap H_{0}^{1}(I))$. For
the remaining term, $F(w,t)-w(t)\mathcal{G}_{0}^{\lambda}(\cdot,0)$, it has
already been noticed that $F(w,t)$ is continuously embedded into $H_{0}%
^{1}(I)$; the same holds for $w(t)\mathcal{G}_{0}^{\lambda}(\cdot,0)$, since
$\mathcal{G}_{0}^{\lambda}\left(  \cdot,0\right)  \in H_{0}^{1}(I)$ and $w$ is
continuous. Moreover, a direct computation of the Fourier coefficients of
$\frac{d^{2}}{dx^{2}}\left[  F(q,t)-q(t)\mathcal{G}_{0}^{\lambda}\left(
\cdot,0\right)  \right]  $ w.r.t. the system $\left\{  \psi_{k}\right\}
_{k\in\mathbb{N}}$ gives%
\begin{equation}
b_{k}=\left\{
\begin{array}
[c]{c}%
\frac{1}{\sqrt{\pi}}\int_{0}^{t}w^{\prime}(s)e^{-i\lambda_{k}\left(
t-s\right)  }\,ds\,-\lambda w(t)\left(  \mathcal{G}_{0}^{\lambda}\left(
\cdot,0\right)  ,\,\psi_{k}\right)  _{L^{2}(I)}\qquad k\ odd\\
0\qquad\qquad\qquad\qquad\qquad\qquad\qquad\qquad\qquad\qquad otherwise
\end{array}
\right.  \label{Fourier coefficients 4}%
\end{equation}
from which it follows%
\begin{equation}
-\frac{d^{2}}{dx^{2}}\left[  F(w,t)-w(t)\mathcal{G}_{0}^{\lambda}\left(
\cdot,0\right)  \right]  =iF(w^{\prime},t)+\lambda w(t)\mathcal{G}%
_{0}^{\lambda}\left(  \cdot,0\right)  \,. \label{phi 4}%
\end{equation}
Then, the estimate (\ref{L1.2 3}) in Lemma \ref{Lemma 1.2} yield: $\frac
{d^{2}}{dx^{2}}\left[  F(q,t)-q(t)\mathcal{G}_{0}^{\lambda}\left(
\cdot,0\right)  \right]  \in C\left(  \left[  0,T\right]  ,\,L_{2}(I)\right)
$ and $\left[  F(w,t)-w(t)\mathcal{G}_{0}^{\lambda}\left(  \cdot,0\right)
\right]  \in C\left(  \left[  0,T\right]  ,\,H^{2}\cap H_{0}^{1}(I)\right)  $,
which implies%
\[
\phi_{t}^{\lambda}\in C\left(  \left[  0,T\right]  ,\,H^{2}\cap H_{0}%
^{1}(I)\right)  \,;
\]
this, together with the boundary condition $-q(t)=\alpha(t)\psi_{t}(0)$
(arising from the equation (\ref{carica 1})), assures that $\psi_{t}\in
D(H_{\alpha(t)})$ at each $t$.

Next we discuss the last point of the Proposition. The continuity in time of
the map $\psi_{t}$ and (\ref{Schroedinger 3}) allow to write: $\psi
(\cdot,t=0)=\psi_{0}$. According to (\ref{ref1})-(\ref{Zeta}), the derivative
$i\partial_{t}\psi_{t}$ is%
\begin{equation}
i\frac{d}{dt}\psi_{t}=-\frac{d^{2}}{dx^{2}}e^{it\Delta_{I}^{D}}\phi
_{0}^{\lambda}+iF(w^{\prime},t)+iZ^{\prime}(t) \label{P2 1}%
\end{equation}%
\[
iZ^{\prime}(t)=-\frac{q(0)}{\sqrt{\pi}}\sum_{k\in D}\frac{\lambda
e^{-i\lambda_{k}t}}{\lambda_{k}+\lambda}\psi_{k}%
\]
On the other hand, from definition (\ref{operatore}) one has%
\begin{equation}
H_{\alpha(t)}\psi_{t}=-\frac{d^{2}}{dx^{2}}\phi_{t}^{\lambda}-\lambda
q(t)\mathcal{G}_{0}^{\lambda}\left(  \cdot,0\right)  \,. \label{P2 2}%
\end{equation}
Making use of (\ref{phi 0.1})-(\ref{Qu}) and (\ref{phi 4}), this writes as%
\[
H_{\alpha(t)}\psi_{t}=-\frac{d^{2}}{dx^{2}}e^{it\Delta_{I}^{D}}\phi
_{0}^{\lambda}+iF(w^{\prime},t)+\lambda q(0)\mathcal{G}_{0}^{\lambda}\left(
\cdot,0\right)  -q(0)\frac{d^{2}}{dx^{2}}Q(t)
\]
with%
\[
-q(0)\frac{d^{2}}{dx^{2}}Q(t)=iF(w^{\prime},t)+\lambda q(0)\mathcal{G}%
_{0}^{\lambda}\left(  \cdot,0\right)  \,.
\]
Combining the above relations, it follows%
\[
i\frac{d}{dt}\psi_{t}-H_{\alpha(t)}\psi(x,t)=0\quad\text{in }L^{2}(I).
\]

\end{proof}

\section{\label{Sec_control}Stability and local controllability}

The evolution problem (\ref{Schroedinger}) defines a map $\Gamma(\alpha
,\psi_{0})$ associating to the coupling parameter, $\alpha$, and the initial
state, $\psi_{0}$, the solution at time $T$. According to the notation
introduced in (\ref{U}) and (\ref{F}), it writes as%
\begin{equation}
\Gamma(\alpha)=e^{iT\Delta}\psi_{0}+F\left(  V(\alpha),T\right)  \,,
\label{G-def1}%
\end{equation}%
\begin{equation}
V(\alpha)=q:q(t)=-\alpha(t)\left(  e^{it\Delta_{I}^{D}}\psi_{0}(0)+\frac
{i}{\pi}Uq(t)\,\right)  \,. \label{G-def2}%
\end{equation}
The local controllability in time $T$ of the dynamics described by
(\ref{Schroedinger}), when $\alpha$ is considered as a control parameter, is
connected with the following question: for $\psi_{0}\in D(H_{\alpha(0)})$,
does $\alpha$ exists such that $\Gamma(\alpha,\psi_{0})=\psi_{T}$ for
arbitrary $\psi_{T}$ in a neighbourhood of $\psi_{0}$? In what follows, we fix
the initial state $\psi_{0}$ and discuss the regularity of the map
$\alpha\rightarrow\Gamma(\alpha,\psi_{0})$. Hence, this control problem will
be considered in the local setting where $\psi_{0}$ concides with a steady
state of the Dirichlet Laplacian and $\alpha$ is small in $H_{0}^{1}(0,T)$. To
simplify the notation, we use $\Gamma(\alpha)$ instead of $\Gamma(\alpha
,\psi_{0})$.

\subsection{Regularity of $\Gamma$}

Recall a standard result in the theory of differential calculus in Banach
spaces. Let $X$ and $Y$ denote two Banach spaces, $W$ an open subset of $X$,
$d_{\alpha}F$ and $d_{\alpha}^{G}F$ respectively the Fr\'{e}chet and the
G\^{a}teaux derivatives of $F:W\rightarrow Y$ evaluated in the point $\alpha$.
A functional $F:W\rightarrow Y$ is of class $C^{1}$ if the map%
\begin{equation}
F^{\prime}:W\rightarrow\mathcal{L}(X,Y),\ F^{\prime}(\alpha)=d_{\alpha}F
\label{C1_def}%
\end{equation}
is continuous.

\begin{theorem}
[(Theorem 1.9 in \cite{Ambrosetti})]\label{prodi}Suppose $F:W\rightarrow Y$ is
G\^{a}teaux-differentiable in $U$. If the map:%
\[
F_{G}^{\prime}:W\rightarrow\mathcal{L}(X,Y),\ F_{G}^{\prime}(\alpha
)=d_{\alpha}^{G}F
\]
is continuous at $\alpha^{\ast}\in W$, then $F$ is Fr\'{e}chet-differentiable
at $\alpha^{\ast}$ and its Fr\'{e}chet derivative evaluated in $\alpha^{\ast}$
results%
\[
d_{\alpha^{\ast}}F=d_{\alpha^{\ast}}^{G}F
\]

\end{theorem}

\begin{lemma}
\label{Lemma 4}The map $V$ defined by (\ref{G-def2}) is $C^{1}(H^{1}%
(0,T);\,H^{1}(0,T))$.
\end{lemma}

\begin{proof}
The continuity of $V$ has been discussed in Lemma \ref{Lemma 1.3}. Next we
consider the differential map $V^{\prime}:H^{1}(0,T)\rightarrow\mathcal{L}%
(H^{1}(0,T),\,H^{1}(0,T))$. Due to Theorem \ref{prodi}, it is enough to prove
that $\alpha\rightarrow d_{\alpha}^{G}V$ is continuous w.r.t. the operator
norm in $\mathcal{L}(H^{1}(0,T),\,H^{1}(0,T))$. The action of $d_{\alpha}%
^{G}V$ on $u\in H^{1}(0,T)$ is%
\begin{equation}
d_{\alpha}^{G}V(u)=q\,, \label{dV_def1}%
\end{equation}%
\begin{equation}
q=-u\,\frac{V(\alpha)}{\alpha}-\alpha\left(  u(0)e^{it\Delta_{I}^{D}}\left.
\mathcal{G}_{0}^{\lambda}(\cdot,0)\right\vert _{x=0}+\frac{i}{\pi}Uq\right)
\,. \label{dV_def2}%
\end{equation}
Then, the $H^{1}$-bound%
\begin{equation}
\left\Vert q\right\Vert _{H^{1}(0,T)}\lesssim\left\Vert u\right\Vert
_{H^{1}(0,T)} \label{bound1.1}%
\end{equation}
and the continuous dependence from $\alpha$ follows from Lemma \ref{Lemma 1.3}
and proceeding as in Corollary \ref{Coroll_1.3_1}.
\end{proof}

\begin{proposition}
\label{Proposition 4}For $\psi_{0}\in D(H_{\alpha(0)})$, the map
$\alpha\rightarrow\Gamma(\alpha)$ is of class $C^{1}(H^{1}(0,T),\,H_{0}%
^{1}(I))$. Moreover, the regular part of the solution at time $T$,
$\Gamma(\alpha)-V(\alpha)(T)\mathcal{G}_{0}^{\lambda}$, considered as a
function of $\alpha$, is of class $C^{1}(H^{1}(0,T),\,H^{2}\cap H_{0}^{1}(I))$.
\end{proposition}

\begin{proof}
The continuity of $\Gamma(\alpha)$ follows directly from the continuity of the
map $\alpha\rightarrow V(\alpha)$ and the estimate (\ref{bound1.0}) in Lemma
\ref{Lemma 1.2}. Next consider the differential map $\Gamma^{\prime}%
:H^{1}(0,T)\rightarrow\mathcal{L}(H^{1}(0,T),\,H_{0}^{1}(I))$. Our aim is to
prove that $\Gamma^{\prime}$ is continuous in the operator norm. Due to the
Theorem \ref{prodi}, it is enough to prove that $\alpha\rightarrow d_{\alpha
}^{G}\Gamma$ is continuous. The action of $d_{\alpha}^{G}\Gamma$ on $u\in
H^{1}(0,T)$ is: $d_{\alpha}^{G}\Gamma(u)=F(d_{\alpha}^{G}V(u),T)$; for
$\alpha,\bar{\alpha}\in H^{1}(0,T)$ the difference: $d_{\alpha}^{G}%
\Gamma-d_{\bar{\alpha}}^{G}\Gamma$ is expressed by%
\[
d_{\alpha}^{G}\Gamma-d_{\bar{\alpha}}^{G}\Gamma=F(d_{\alpha}^{G}%
V(u)-d_{\bar{\alpha}}^{G}V(u),T)
\]
Using (\ref{bound1.0}), one has%
\begin{multline*}
\left\vert \left\vert \left\vert d_{\alpha}^{G}\Gamma-d_{\bar{\alpha}}%
^{G}\Gamma\right\vert \right\vert \right\vert =\sup_{\substack{u\in
H^{1}(0,T)\\\left\Vert u\right\Vert _{H^{1}(0,T)}=1}}\left\Vert F(d_{\alpha
}^{G}V(u)-d_{\bar{\alpha}}^{G}V(u),T)\right\Vert _{H_{0}^{1}(I)}\lesssim\\
\lesssim\sup_{\substack{u\in H^{1}(0,T)\\\left\Vert u\right\Vert _{H^{1}%
(0,T)}=1}}\left\Vert d_{\alpha}^{G}V(u)-d_{\bar{\alpha}}^{G}V(u)\right\Vert
_{H^{1}(0,T)}=\left\vert \left\vert \left\vert d_{\alpha}^{G}V-d_{\bar{\alpha
}}^{G}V\right\vert \right\vert \right\vert
\end{multline*}
then, the continuity of $\alpha\rightarrow d_{\alpha}^{G}\Gamma$ follows from
the continuity of $\alpha\rightarrow d_{\alpha}^{G}V$ proved in Lemma
\ref{Lemma 4}.

Let introduce the map%
\begin{equation}
\mu(\alpha)=\Gamma(\alpha)-V(\alpha)(T)\mathcal{G}_{0}^{\lambda}(\cdot,0)
\label{mu_alpha}%
\end{equation}
representing the regular part of the state at time $T$. Combining the result
of Lemma \ref{Lemma 4} and the first part of this proof, one has: $\mu
(\alpha)\in C^{1}(H^{1}(0,T),\,H_{0}^{1}(I))$. To achive our result it remains
to prove that: $-\partial_{x}^{2}\mu(\alpha)$ belongs to $C^{1}(H^{1}%
(0,T),\,L^{2}(I))$. According to (\ref{phi 0.1})-(\ref{Q(t)}), this can be
written as%
\[
-\frac{d^{2}}{dx^{2}}\left[  e^{iT\Delta_{I}^{D}}\phi_{0}^{\lambda
}+F(w,T)-w(T)\mathcal{G}_{0}^{\lambda}(\cdot,0)+\left[  V(\alpha)\right]
(0)Q(T)\right]
\]
where $w(t)=\left[  V(\alpha)\right]  (t)-\left[  V(\alpha)\right]  (0)$,
while the regularity of the terms: $-\frac{d^{2}}{dx^{2}}e^{iT\Delta_{I}^{D}%
}\phi_{0}^{\lambda}$ and $Q(T)$ is discussed in Proposition
\ref{Proposition 3}. For $\alpha,\tilde{\alpha}\in H^{1}(0,T)$, the difference
$\partial_{x}^{2}\left(  \mu(\alpha)-\mu(\tilde{\alpha})\right)  $ is%
\begin{align*}
\mu(\alpha)-\mu(\tilde{\alpha})  &  =-\frac{d^{2}}{dx^{2}}\left[  \left(
F\left(  w,T\right)  -F\left(  \tilde{w},T\right)  \right)  -\left(
w(T)-\tilde{w}(T)\right)  \mathcal{G}_{0}^{\lambda}\right]  +\\
&  -\left(  \left[  V(\alpha)\right]  (0)-\left[  V(\tilde{\alpha})\right]
(0)\right)  \frac{d^{2}}{dx^{2}}Q(T)\,.
\end{align*}
Using (\ref{phi 4}), one has%
\begin{align*}
\mu(\alpha)-\mu(\tilde{\alpha})  &  =i\left(  F\left(  w^{\prime},T\right)
-F\left(  \tilde{w}^{\prime},T\right)  \right)  +\lambda\left(  w(T)-\tilde
{w}(T)\right)  \mathcal{G}_{0}^{\lambda}+\\
&  -\left(  \left[  V(\alpha)\right]  (0)-\left[  V(\tilde{\alpha})\right]
(0)\right)  \frac{d^{2}}{dx^{2}}Q(T)
\end{align*}
$w^{\prime}$ and $\tilde{w}^{\prime}$ respectively denoting the time
derivatives of the functions $V(\alpha)$ and $V(\tilde{\alpha})$. The
continuity of $\mu(\alpha)$, then, follows from the continuity of $F(\cdot,T)$
(expressed by (\ref{bound1.0})) and of the map $V(\cdot)$. The G\^{a}teaux
derivative $d_{\alpha}^{G}\mu$ acts on $u\in H^{1}(0,T)$ as%
\[
d_{\alpha}^{G}\mu(u)=-\frac{d^{2}}{dx^{2}}\left[  F(d_{\alpha}^{G}%
w(u),T)-\left[  d_{\alpha}^{G}w(u)\right]  (T)\mathcal{G}_{0}^{\lambda}%
(\cdot,0)+\left[  d_{\alpha}^{G}V(\alpha)\right]  (0)Q(T)\right]  \,.
\]
Proceeding as before, the difference $d_{\alpha}^{G}\mu(u)-d_{\tilde{\alpha}%
}^{G}\mu(u)$ is%
\begin{align*}
d_{\alpha}^{G}\mu(u)-d_{\tilde{\alpha}}^{G}\mu(u)  &  =i\left(  F\left(
\partial_{t}d_{\alpha}^{G}w(u)-\partial_{t}d_{\tilde{\alpha}}^{G}%
w(u),T\right)  \right)  +\lambda\left(  \left[  d_{\alpha}^{G}w(u)\right]
(T)-\left[  d_{\tilde{\alpha}}^{G}w(u)\right]  (T)\right)  \mathcal{G}%
_{0}^{\lambda}+\\
&  -\left(  \left[  d_{\alpha}^{G}w(u)\right]  (0)-\left[  d_{\tilde{\alpha}%
}^{G}w(u)\right]  (0)\right)  \frac{d^{2}}{dx^{2}}Q(T)\,.
\end{align*}
Using (\ref{bound1.0}) and recalling, from Lemma \ref{Lemma 4}, that
$\alpha\rightarrow d_{\alpha}^{G}V$ is continuous in the $\mathcal{L}\left(
H^{1}(0,T),\,H^{1}(0,T)\right)  $ operator norm, this relation leads to
continuity of $\alpha\rightarrow d_{\alpha}^{G}\mu$ in the $\mathcal{L}\left(
H^{1}(0,T),\,L^{2}(I)\right)  $-norm.
\end{proof}

\subsection{\label{Sec_control_lin}The linearized map}

Next we discuss the surjectivity of the linearized map $d_{\alpha}^{G}\Gamma$
for $\alpha\in H_{0}^{1}(0,T)$. Since the transformations $V$ and $d_{\alpha
}^{G}V$ preserves Dirichlet boundary conditions, in this framework the result
of Lemma \ref{Lemma 4} raphrases as:

\begin{lemma}
\label{Lemma 4.1}The map $V$ defined by (\ref{G-def2}) is $C^{1}(H_{0}%
^{1}(0,T);\,H_{0}^{1}(0,T))$.
\end{lemma}

This also implies that the state at the initial and finial times possesses
only regular part. Thus, Proposition \ref{Proposition 4} become:

\begin{proposition}
\label{Proposition 4.1}For $\psi_{0}\in H^{2}\cap H_{0}^{1}(I)$, the map
$\alpha\rightarrow\Gamma(\alpha)$ is of class $C^{1}(H_{0}^{1}(0,T),\,H^{2}%
\cap H_{0}^{1}(I))$.
\end{proposition}

It is whorthwhile to notice that, when $\psi_{0}$ is given by a linear
superposition of eigenstates of odd kind (i.e. $\psi_{k}=\sin\frac{k}%
{2}x,\ k\ $even), the source term in (\ref{G-def2}), $e^{it\Delta}\psi_{0}%
(0)$, is null at each $t$, and the charge, $V(\alpha)$ in the above notation,
is identically zero. In these conditions, the particle does not 'feel' the
interaction and the evolution is simply determined by the free propagator
$e^{it\Delta_{I}^{D}}$. This implies: $\Gamma(\alpha)=e^{iT\Delta_{I}^{D}}%
\psi_{0}$. In order to avoid this situation, we assume the initial state
$\psi_{0}$ to be fixed in the subspace of even functions%
\begin{equation}
\mathcal{W}=\left\{  \varphi\in H^{2}\cap H_{0}^{1}(I)\,\left\vert
\,\varphi=\sum_{k\in D}c_{k}\psi_{k}\right.  \right\}
\label{stato iniziale 1}%
\end{equation}
This choice also implies that $\Gamma$ takes values in $\mathcal{W}$, as can
be cheked by its explicit formula. Due to Proposition \ref{Proposition 4.1},
one has: $\Gamma\in C^{1}(H_{0}^{1}(0,T),\,\mathcal{W})$.

The linearized map $d_{\alpha}^{G}\Gamma$ is defined according to%
\begin{equation}
d_{\alpha}^{G}\Gamma(u)=\frac{i}{\sqrt{\pi}}\sum\limits_{k\in D}\left(
\int_{0}^{T}q(s)e^{-i\lambda_{k}\left(  T-s\right)  }\,ds\right)  \,\psi_{k}
\label{Lin1}%
\end{equation}%
\begin{equation}
-q(t)=\frac{i}{\pi}\alpha(t)Uq(t)+u(t)\left(  e^{it\Delta}\psi_{0}(0)+\frac
{i}{\pi}U\left[  V(\alpha)\right]  (t)\right)  \label{Lin2}%
\end{equation}
Let $\psi_{T}\in\mathcal{W}$ be a target state for the linear transformation
(\ref{Lin1})-(\ref{Lin2}). We address the following question: does $u\in
H_{0}^{1}(0,T)$ exists such that $d_{\alpha}^{G}\Gamma(u)=\psi_{T}$ ?\newline
Consider at first the equation%
\begin{equation}
\psi_{T}=\frac{i}{\sqrt{\pi}}\sum\limits_{k\in D}\left(  \int_{0}^{T}%
\rho(s)e^{-i\lambda_{k}\left(  T-s\right)  }\,ds\right)  \,\psi_{k}
\label{Lin1.1}%
\end{equation}
w.r.t. the variable $\rho$. Denoting with $c_{k}$ the Fourier coefficients of
$\psi_{T}$ in the basis $\left\{  \psi_{k}\right\}  _{k\in D}$, the above
equation is equivalent to%
\begin{equation}
c_{k}=\frac{i}{\sqrt{\pi}}\int_{0}^{T}\rho(s)e^{-i\lambda_{k}\left(
T-s\right)  }\,ds,\qquad k\in D \label{Lin1.2}%
\end{equation}
This is a trigonometric moment problem; the existence of solutions is
established by K. Beauchard and C. Laurent in \cite{BeLa} for arbitrary values
of $T>0$ according to the following result.

\begin{theorem}
\label{Th_K}Let $T>0$ and $\left\{  \omega_{k}\right\}  _{k\in\mathbb{N}_{0}}%
$, with $\mathbb{N}_{0}=\mathbb{N\cup}\left\{  0\right\}  $, be an increasing
sequence of $\left[  0,+\infty\right)  $ such that $\omega_{0}=0$ and%
\[
\lim_{k\rightarrow+\infty}\left(  \omega_{n+1}-\omega_{n}\right)  =+\infty
\]
For every $c\in\ell^{2}(\mathbb{N}_{0},\mathbb{C})$, there exist infinitely
many $v\in L^{2}\left(  \left(  0,T\right)  ,\mathbb{C}\right)  $ solving%
\begin{equation}
c_{k}=\int_{0}^{T}v(s)e^{i\omega_{n}s}\,ds,\qquad n\in\mathbb{N}_{0}\,.
\label{moment_eq}%
\end{equation}

\end{theorem}

\begin{proof}
The proof rephrases the one given in Corollary 1 of \cite{BeLa}, where the
authors discuss the existence in $L^{2}\left(  \left(  0,T\right)
,\mathbb{R}\right)  $ under the particular condition: $c\in\ell^{2}%
(\mathbb{N}_{0},\mathbb{C})$, $c_{0}\in\mathbb{R}$. For sake of completeness
we give a sketch of it.

Let define: $\omega_{-n}=-\omega_{n}$ for $n\in\mathbb{N}_{0}$. According to
the result of \cite{Haraux} (referred as Theorem 1 in \cite{BeLa}), the family
$\left\{  e^{i\omega_{n}t}\right\}  _{n\in\mathbb{Z}}$ is a Riesz basis of the
subspace $\mathcal{F}$ defined by the clousure in $L^{2}(0,T)$ of the
$Span\left\{  e^{i\omega_{n}t}\right\}  _{n\in\mathbb{Z}}$. This condition is
equivalent to the existence of an isomorphism $J$ (we refer to Proposition 20
in \cite{BeLa}):%
\[%
\begin{array}
[c]{llll}%
\medskip J: & \mathcal{F} & \mathcal{\longrightarrow} & \ell^{2}%
(\mathbb{Z},\mathbb{C})\\
& v & \mathcal{\longrightarrow} & \int_{0}^{T}v(s)e^{i\omega_{n}s}\,ds
\end{array}
\]
Given $c\in\ell^{2}(\mathbb{N}_{0},\mathbb{C})$, we consider $\tilde{c}\in
\ell^{2}(\mathbb{Z},\mathbb{C})$ such that: $\tilde{c}_{k}=c_{k}$ for
$k\in\mathbb{N}_{0}$. To any choice of $\tilde{c}$ there corresponds an unique
solution to (\ref{moment_eq}) defined by $J^{-1}(\tilde{c})$.
\end{proof}

When the $\lambda_{k}$ coincide with the frequencies of the standard basis
$\left\{  e^{i\omega nt},\ \omega=\frac{2\pi}{T}\right\}  _{n\in\mathbb{Z}}$
in $L^{2}(0,T)$, equation (\ref{Lin1.2}) can be also interpreted as a
constraint over the Fourier coefficients of the function $\rho$. In this case,
it is possible to give a solution of the problem respecting Dirichlet boundary
conditions. These remarks are summarized in the following proposition.

\begin{proposition}
\label{Prop_moment}The following properties hold:\medskip\newline$1)$ For
$\left\{  c_{k}\right\}  \in\ell_{2}(D)$ and $T>0$, (\ref{Lin1.2}) admits
infinitely many solutions $\rho\in L^{2}(0,T)$.\medskip\newline$2)$ Let
$T\geq8\pi$. For any $\psi_{T}\in\mathcal{W}$, it is possible to find at least
one $\rho\in H_{0}^{1}(0,T)$ solving (\ref{Lin1.1}).
\end{proposition}

\begin{proof}
$1)$ Let $\psi_{T}=\sum_{k\in D}c_{k}\psi_{k}$, and consider the corresponding
moment equation (\ref{Lin1.2}). For $k\in D$, we set: $k=2n-1$, $\omega
_{n}=\lambda_{2n-1}$ with $n\in\mathbb{N}$, and $\omega_{0}=0$; with this
notation, our problem writes as%
\begin{equation}
\tilde{c}_{n}=\int_{0}^{T}\tilde{\rho}(s)e^{i\omega_{n}s}\,ds,\qquad
n\in\mathbb{N}_{0} \label{control_1}%
\end{equation}%
\begin{equation}
\tilde{\rho}=\frac{i}{\sqrt{\pi}}\rho\,,\qquad\left.  \tilde{c}_{n,T}%
\medskip\right\vert _{n\in\mathbb{N}}=e^{i\lambda_{k}T}c_{2n-1}\quad
\text{and\quad}\tilde{c}_{0}=0 \label{control_1.1}%
\end{equation}
According to theorem \ref{Th_K}, (\ref{control_1}) admits infinitely many
solutions $\tilde{\rho}\in L^{2}(0,T)$, which are determined by fixing the
extensions of the families $e^{i\omega_{n}s}$ and $\tilde{c}_{n,T}$ to
$n\in\mathbb{Z}$.

$2)$ For $T=8\pi$, the functions $\left\{  e^{-i\lambda_{k}t}\right\}  _{k\in
D}$ form a subsystem of the standard basis $\left\{  e^{i\omega nt}%
,\ \omega=\frac{2\pi}{T}\right\}  _{n\in\mathbb{Z}}$ in $L^{2}(0,T)$. Thus,
the equation (\ref{Lin1.2}) partially determines the Fourier coefficients of
the function $\rho$. A particular solution $\rho_{0}$ is obtained by a
superposition respecting Dirichlet conditions on the boundary%
\begin{equation}
\frac{i}{\sqrt{\pi}}\rho_{0}(t)=\sum_{k\in D}e^{i\lambda_{k}T}c_{k}\left(
e^{-i\lambda_{k}t}-e^{i\lambda_{k}t}\right)  \,. \label{solution}%
\end{equation}
Since the regularity of $\psi_{T}\in\mathcal{W}$ implies: $\left\{
\,k^{2}c_{k}\right\}  \in\ell_{2}(D)$, it follows $\rho\in H_{0}^{1}(0,T)$.
For $T>8\pi$, the target $\psi_{T}$ is attained by $\rho$ such that: $\left.
\rho\right\vert _{t\in\left[  0,8\pi\right]  }=\rho_{0}$ is a solution of
(\ref{Lin1.1}) in $H_{0}^{1}(0,8\pi)$ and $\rho(t\geq8\pi)=0$.
\end{proof}

In order to solve the inverse problem related to (\ref{Lin1})-(\ref{Lin2}), we
proceed as follows: for a target state $\psi_{T}\in\mathcal{W}$, Proposition
\ref{Prop_moment} allows to determine (at least one) $q\in H_{0}^{1}(0,T)$
solving (\ref{Lin1}); then, this function is replaced in (\ref{Lin2}) in the
attempt of finding a suitable $u\in H_{0}^{1}(0,T)$ such that%
\begin{equation}
-q(t)=\frac{i}{\pi}\alpha(t)Uq(t)+u(t)\left(  e^{it\Delta}\psi_{0}(0)+\frac
{i}{\pi}U\left[  V(\alpha)\right]  (t)\right)
\end{equation}
The solvability of this equation w.r.t. $u$ is strictly connected with the
specific choice of the initial state, $\psi_{0}$, and of the linearization
point $\alpha$. In particular, if the function $\left(  e^{it\Delta}\psi
_{0}(0)+\frac{i}{\pi}U\left[  V(\alpha)\right]  (t)\right)  $ is different
from zero at each $t$, one can easly determine $u$ as%
\begin{equation}
u(t)=\left(  -q(t)-\frac{i}{\pi}\alpha(t)Uq(t)\right)  \left(  e^{it\Delta
}\psi_{0}(0)+\frac{i}{\pi}U\left[  V(\alpha)\right]  (t)\right)  ^{-1}%
\end{equation}
Consider a simplified framework where $\psi_{0}$ is a Laplacian's eigenstate
in $\mathcal{W}$%
\begin{equation}
\psi_{0}=\psi_{\bar{k}}\quad\bar{k}\ odd \label{stato iniziale 3}%
\end{equation}
and $\alpha=0$. In this setting, the equation (\ref{Lin2}) writes as%
\begin{equation}
-q(t)=\frac{1}{\sqrt{\pi}}u(t)\,e^{-i\lambda_{\bar{k}}t} \label{Lin4}%
\end{equation}
and the inverse problem%
\begin{equation}
\psi_{T}=-\frac{i}{\pi}\sum\limits_{k\in D}\left(  \int_{0}^{T}%
u(s)\,e^{-i\lambda_{\bar{k}}s}e^{-i\lambda_{k}\left(  T-s\right)
}\,ds\right)  \,\psi_{k} \label{Lin3}%
\end{equation}
is solved in $H_{0}^{1}(0,T)$ by the last point of Proposition
\ref{Prop_moment}. This result, and the regularity of the map $\Gamma$ leads
to the local steady state controllability of the system (\ref{G-def1}%
)-(\ref{G-def2}) around the point $\alpha=0$.

\begin{theorem}
[Local controllability]\label{local controllability}Let $T\geq8\pi$ and assume
(\ref{stato iniziale 3}). It exist an open neighbourhood $V\times P\subseteq
H_{0}^{1}(0,T)\times\mathcal{W}$ of the point $\left(  0,\psi_{\bar{k}%
}\right)  $ such that: $\left.  \Gamma\right\vert _{V}:V\rightarrow P$ is surjective
\end{theorem}

\begin{proof}
As remarked at the beginning of this section, for $\psi_{0}\in\mathcal{W}$,
the map $\Gamma$ is $C^{1}(H_{0}^{1}(0,T),\,\mathcal{W})$. Moreover, under our
assumptions, $d_{\alpha}\Gamma$ is surjective for $\alpha=0$. Then, by the
Local Surjectivity Theorem (e.g. \cite{Ratiu}) we know that there exists an
open neighbourhood $V\times P\subseteq H_{0}^{1}(0,T)\times\mathcal{W}$ of
$\left(  0,\Gamma(0)\right)  $ such that the restriction $\left.
\Gamma\right\vert _{V}:V\rightarrow P$ is surjective. By definition,
$\Gamma(0)=e^{-i\lambda_{\bar{k}}T}\psi_{\bar{k}}$ and $P$ is a neighbourhood
of $\psi_{\bar{k}}$ in $\mathcal{W}$.
\end{proof}

\end{document}